\documentclass[final]{siamltex}
\usepackage{graphicx,bm,amssymb,amsmath,xcolor} 
\newcommand{\bi}{\begin{itemize}}
\newcommand{\ei}{\end{itemize}}
\newcommand{\ben}{\begin{enumerate}}
\newcommand{\een}{\end{enumerate}}
\newcommand{\be}{\begin{equation}}
\newcommand{\ee}{\end{equation}}
\newcommand{\bea}{\begin{eqnarray}} 
\newcommand{\eea}{\end{eqnarray}}
\newcommand{\ba}{\begin{align}} 
\newcommand{\ea}{\end{align}}
\newcommand{\bse}{\begin{subequations}} 
\newcommand{\ese}{\end{subequations}}
\newcommand{\bc}{\begin{center}}
\newcommand{\ec}{\end{center}}
\newcommand{\bfi}{\begin{figure}}
\newcommand{\efi}{\end{figure}}
\newcommand{\ca}[2]{\caption{#1 \label{#2}}}
\newcommand{\ig}[2]{\includegraphics[#1]{#2}}
\newcommand{\bmp}[1]{\begin{minipage}{#1}}
\newcommand{\emp}{\end{minipage}}
\newcommand{\bp}{\begin{proof}}
\newcommand{\ep}{\end{proof}}

\newcommand{\tbox}[1]{{\mbox{\tiny #1}}}

\newcommand{\half}{\mbox{\small $\frac{1}{2}$}}

\newcommand{\al}{\alpha}

\newcommand{\ttau}{\tilde\tau}   
\newcommand{\tsig}{\tilde\sigma}   
\newcommand{\om}{\omega}

\newcommand{\eps}{\epsilon}
\newcommand{\pO}{{\partial\Omega}}
\newcommand{\C}{\mathbb{C}}
\newcommand{\CpO}{\C\backslash\pO}  
\newcommand{\CO}{\C\backslash\overline{\Omega}}  

\newcommand{\R}{\mathbb{R}}
\newcommand{\RR}{\mathbb{R}^2}
\newcommand{\vN}{v^{(N)}}   

\newcommand{\Ny}{Nystr\"om}
\newcommand{\Nf}{M}              
\DeclareMathOperator{\im}{Im}
\DeclareMathOperator{\re}{Re}
\newcommand{\matlab}{MATLAB}               
\newcommand{\mpspack}{{\tt MPSpack}}       
\newtheorem{thm}{Theorem}

\newtheorem{lem}[thm]{Lemma}

\newtheorem{pro}[thm]{Proposition}
\newtheorem{rmk}[thm]{Remark}

\begin{document} 

\title{Evaluation of layer potentials close to the boundary
for Laplace and Helmholtz problems on analytic planar domains}

\author{Alex H. Barnett\thanks{Department of Mathematics, 
Dartmouth College, Hanover, NH, 03755, USA}}

\maketitle
\begin{abstract}
Boundary integral equations are an
efficient and accurate tool for the numerical solution
of 
elliptic boundary value problems.
The solution is expressed as a layer potential; however,
the error in its evaluation grows large near the boundary
if a fixed 
quadrature rule is used.
Firstly, we analyze this error for Laplace's equation with
analytic density and the global periodic trapezoid rule,
and find an
intimate connection to the complexification of the boundary parametrization.
Our main result is then a simple and efficient scheme for accurate
evaluation up to the boundary
for single- and double-layer potentials for the Laplace and Helmholtz equations,
using surrogate local expansions about centers placed near the boundary.
The scheme---which also underlies the recent QBX \Ny\ quadrature---%
is asymptotically
exponentially convergent (we prove this in the analytic Laplace case),
requires no adaptivity,
generalizes simply to three dimensions,
and has $O(N)$ complexity when executed via a
locally-corrected fast multipole sum. 
%
We give an example of high-frequency scattering from an obstacle
with perimeter 700 wavelengths long, evaluating the
solution at 
$2\times 10^5$ points near the boundary
with 11-digit accuracy in 30 seconds in MATLAB on a single CPU core.
\end{abstract}


\section{Introduction}

We are interested in solving
boundary-value problems (BVPs) of the type
\bea
(\Delta + \om^2) u & = & 0 \qquad \mbox{ in }\Omega
\label{bvp1}
\\
u & = & f \qquad \mbox{ on }\pO
\label{bvp2}
\eea
where $\Omega\subset\RR$ is either an interior or exterior domain
with boundary curve $\pO$, and
either $\om=0$ (Laplace equation) or $\om>0$ (Helmholtz equation).
We will mostly use the above Dirichlet boundary condition in our examples,
and note that Neumann and other types of boundary conditions
can equally well benefit from our technique.
The numerical solution of this type of BVP has
numerous applications in electrostatics,
equilibrium problems, and acoustic or electromagnetic
wave scattering in the frequency domain.
The case $f\equiv0$ includes eigenvalue (cavity resonance)
problems for the Laplacian.

The boundary integral approach \cite{atkinson,LIE} has many advantages over
conventional finite element or finite difference discretization of the domain:
very few unknowns are needed since the problem is now of lower dimension,
provable high-order accuracy is simple to achieve,
and, in exterior domains, radiation conditions are automatically
enforced without the use of artificial boundaries.
Thus the approach is especially useful for wave scattering,
including at high frequency \cite{coltonkress,bathreadingrev}.

The integral equation approach exploits
the known fundamental solution for the PDE,
\be
\Phi(x,y) = \left\{\begin{array}{ll}
\frac{1}{2\pi}\log\frac{1}{|x-y|}, & \om=0,\\
\frac{i}{4}H_0^{(1)}(\om|x-y|), & \om>0.
\end{array}\right.
\label{Phi}
\ee
For the interior case, which is the simplest, the BVP \eqref{bvp1}--\eqref{bvp2}
is converted to a boundary integral equation (BIE), which is
of the Fredholm second kind,
\be
(D - \half I)\tau = f~,
\label{bie}
\ee
where $\tau$ is an unknown density function on $\pO$,
$I$ is the identity,
$D:C(\pO)\to C(\pO)$ is the double-layer integral operator with kernel $k(x,y) =
\partial \Phi(x,y)/\partial n(y)$, and $n(y)$ is the outward normal at
$y\in\pO$.
Numerical solution of \eqref{bie}, for instance via the Nystr\"om method
\cite[Ch.~12]{LIE} \cite[Ch.~4]{atkinson} with $N$ quadrature nodes on $\pO$,
results in an approximation to $\tau$ sampled at these nodes,
from which one may recover an approximation to $\tau$
on the whole of $\pO$ by interpolation (e.g.\ \Ny\ interpolation).
Finally, one can evaluate the approximate BVP solution at any target
point $x$ in the domain as the 
double-layer potential
\be
u(x) = \int_\pO \frac{\partial \Phi(x,y)}{\partial n(y)}\tau(y) ds_y ~,
\qquad x\in\Omega~.
\label{dlp}
\ee
It is convenient, and common practice, to
reuse the {\em existing} $N$ quadrature nodes underlying the
Nystr\"om method to approximate the integral \eqref{dlp}---in other
words, to skip the interpolation step;
we will call this the
{\em native} evaluation scheme.\footnote{In
    \cite{qbx} this is called the ``underlying'' scheme, and in
    \cite{helsing_close} ``straight-up'' quadrature.}
It is common wisdom that this gives an accurate
solution when $x$ is ``far'' from $\pO$, but a very inaccurate one
close to $\pO$, even when the $\tau$ samples themselves are accurate.
Fig.~\ref{f:err} (a),
whose content will be familiar to anyone who
has 
tested the accuracy of a BIE method,
illustrates this: the error of evaluation
grows to $O(1)$ as one nears $\pO$.
(Also see \cite[Fig.~2]{helsing_close} or \cite[Fig.~4]{qbx} which show
a similar story for a panel-based underlying quadrature.)

\bfi 
\bmp{3in}
(a) {\small $\log_{10} |u^{(N)}-u|$\quad DLP, $\tau\equiv 1$, \;$\om=0$}\\
\ig{width=2.8in}{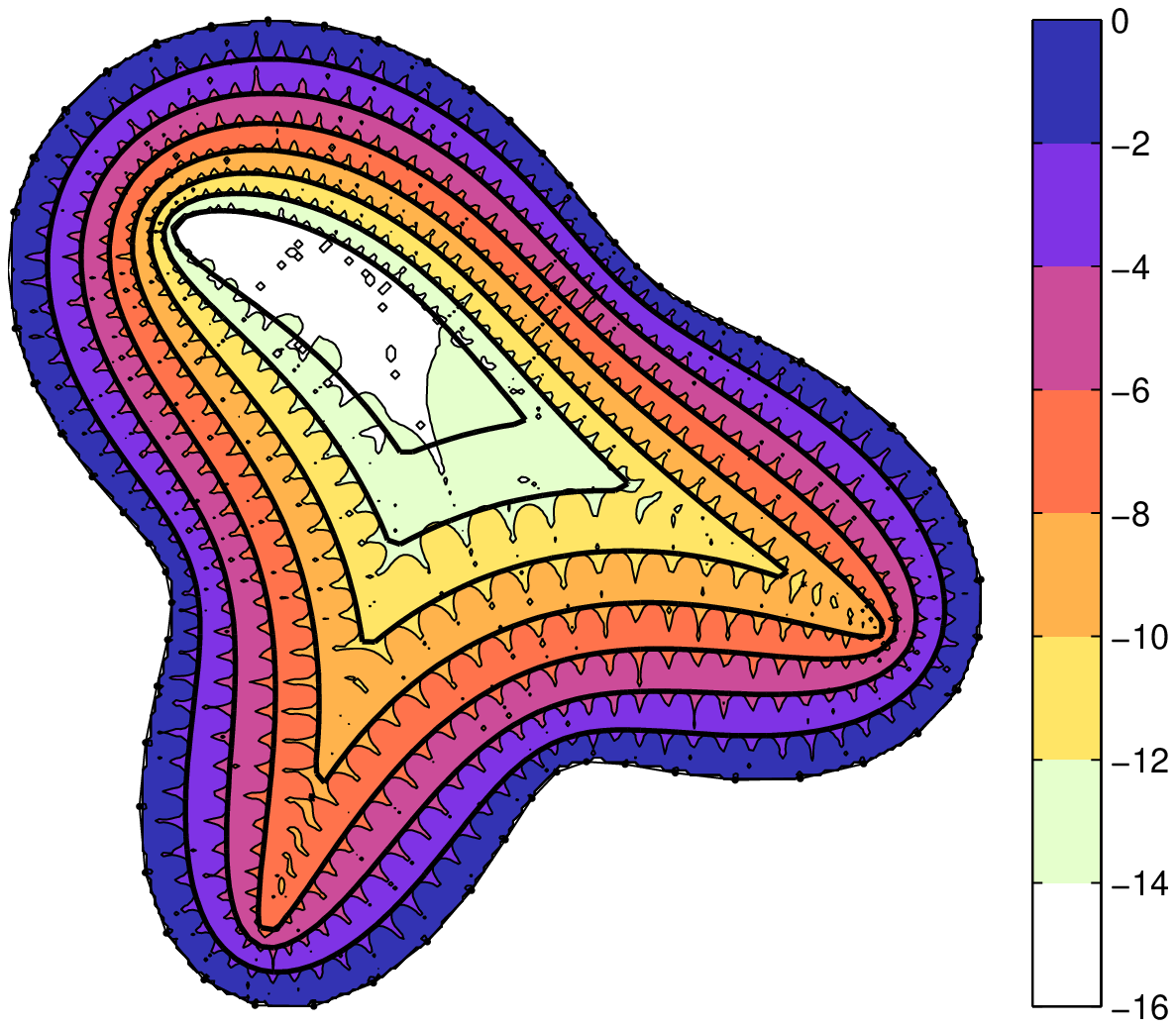}
\emp
\bmp{3in}
(b) {\small $\log_{10} |u^{(N)}-u|$ \quad DLP, BVP $u(x,y) = xy$,\;$\om=0$}\\
\ig{width=2.8in}{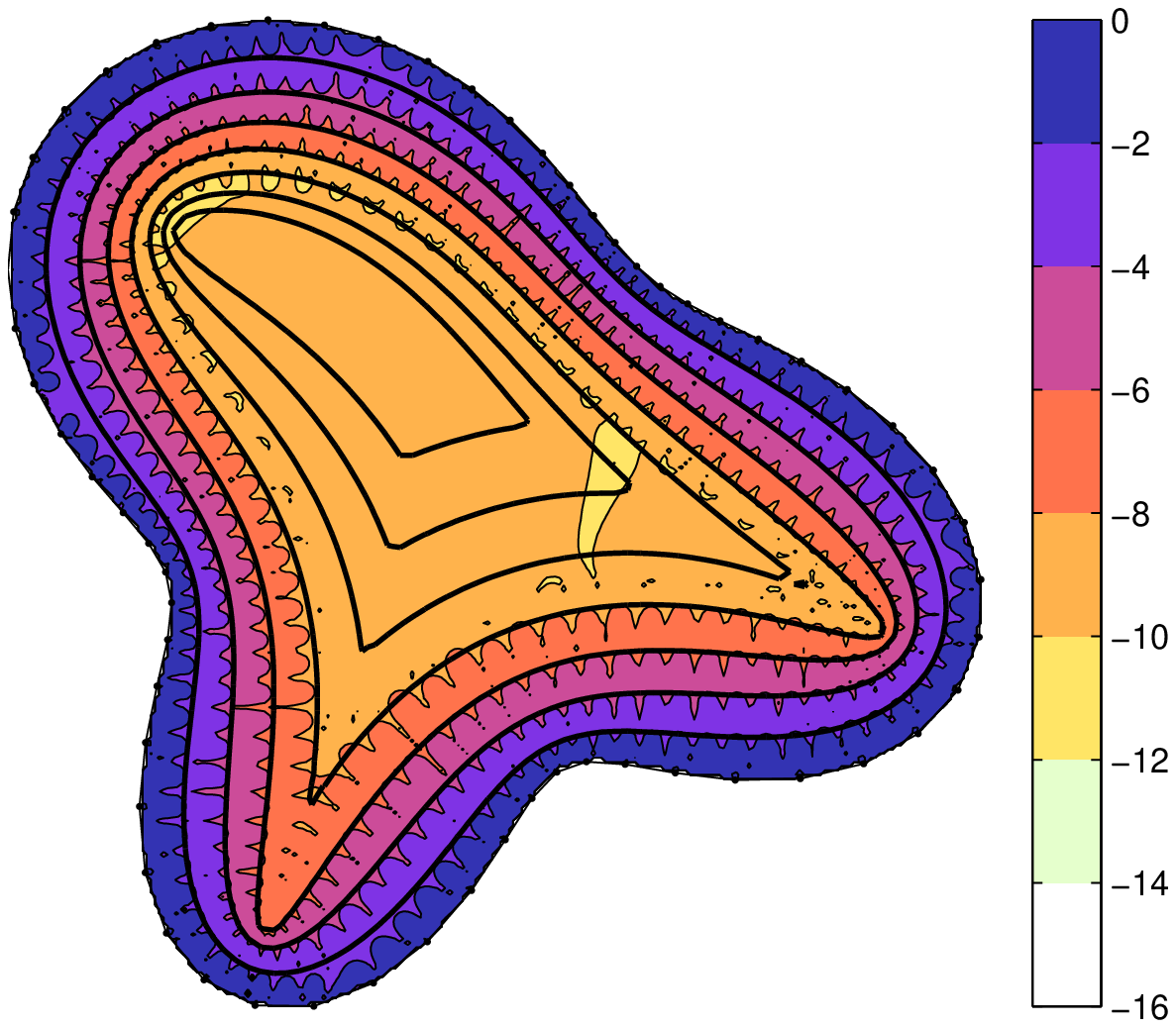}
\emp
\vspace{2ex}
\bmp{3in}
(c) {\small $\log_{10} |u^{(N)}-u|$ \quad DLP, $\tau\equiv 1$, \;$\om=0$}\\
\ig{width=2.8in}{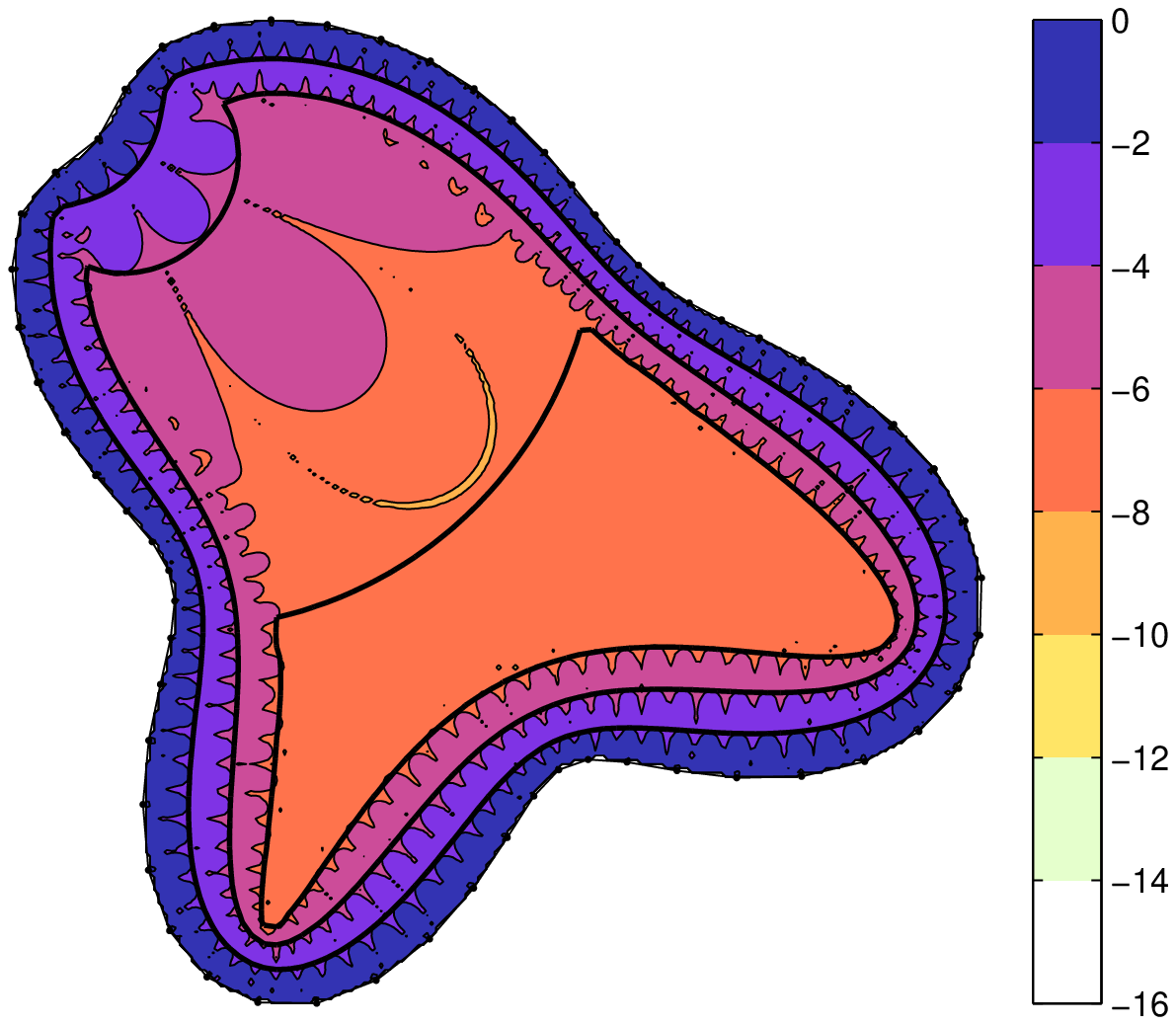}
\emp
\bmp{3in}
(d) {\small $\log_{10} |u^{(N)}-u|$ \quad GRF, $u$ point src, $\om=2$} \\
\ig{width=2.8in}{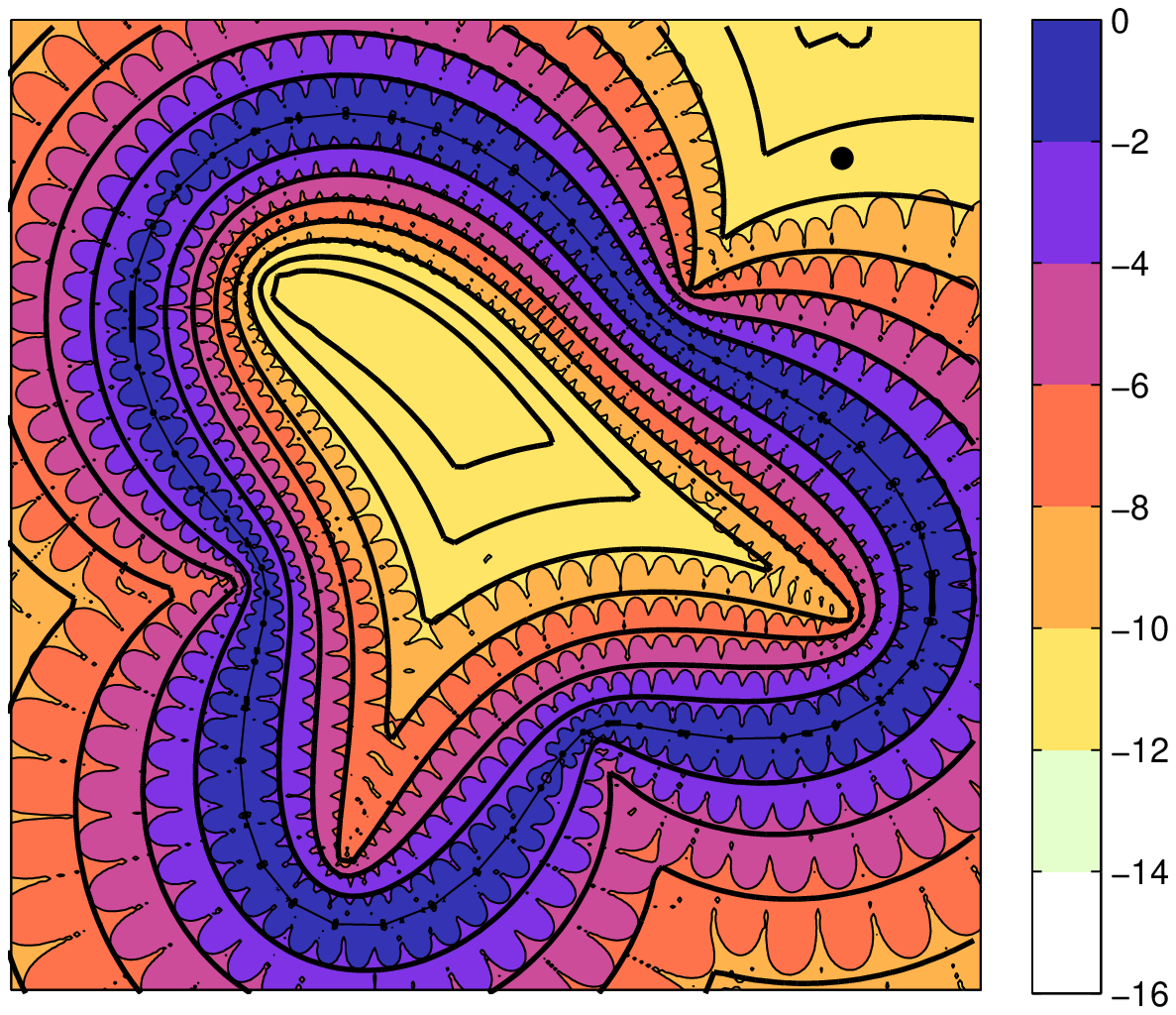}
\emp
\ca{Evaluation error for Laplace and Helmholtz layer potentials
using the native quadrature scheme described after \eqref{dlp}.
In all four cases there are $N=60$ nodes.
Error contours separated by a factor of $10^2$
are shown (thin black lines),
as are the contours (thick black lines) predicted by Theorem~\ref{t:convd}.
In (a), (b), and (d) the interior domain $\Omega$
has boundary given by the polar function
$r(\theta) = 1 + 0.3 \cos[3(\theta+ 0.3\sin\theta)]$.
(a) is a double-layer with $\tau\equiv 1$ and $\om=0$.
(b) $\tau$ is the Nystr\"om solution for Dirichlet
data corresponding to the potential $u(x,y) = xy$, with $\om=0$.
(c) same as (a) but $\pO$ has
a small Gaussian ``bump'' on its northwest side.
(d) Test of Green's representation formula \eqref{grf}, both inside and outside,
for $\om=2$ (roughly 0.8 wavelength across the diameter)
with data due to an exterior point source at the dot shown.
}{f:err}
\efi

There are several ways to overcome this problem near the boundary.
In order of increasing sophistication, these include:
i) increase $N$ in the Nystr\"om method (although
solving a linear system larger than one needs is clearly a
waste of resources);
%
ii) fix $N$, but then interpolate $\tau$ onto a finer fixed set of boundary
nodes, enabling 
points closer to $\pO$ to be accurately evaluated
\cite[Fig. 7.4]{atkinson}
(this is implemented in \cite{jeon98}---what is not
discussed is that the number of finer nodes must grow without limit
as $x \to \pO$);
%
iii) use an adaptive quadrature scheme for \eqref{dlp}
which is able to access the interpolant of $\tau$ \cite{gedney03}
(although this
achieves high accuracy for each point $x\in\Omega$, we have found it
very slow \cite[Sec.~5.1]{qplp}
because 
the adaptivity depends on $x$, with an arbitrarily large
number of refinements needed as $x\to\pO$);
iv) use various fixed-order methods based upon precomputed
quadratures \cite{mckenney} or grids \cite{mayo}
for the Laplace case;
or v) use high-order methods of Helsing--Ojala \cite{helsing_close}
for the Laplace case,
which approach machine precision accuracy
while maintaining efficiency within a fast multipole (FMM) accelerated
scheme.


The method of \cite{helsing_close}
has recently been extended to the Helmholtz equation in
two dimensions \cite{helsingtut}.
Here we present an alternative, simple, and efficient new method
that addresses the close-to-boundary quadrature problem, with numerical
effort independent of the distance of $x$ from $\pO$.
One advantage is that our scheme extends naturally to
the three-dimensional case, unlike existing high-order two-dimensional
schemes. Incidentally, our scheme equally well evaluates the potential {\em on
the curve} $\pO$ (i.e.\ the limit of $x$ approaching $\pO$ from one side),
meaning that it can also be used to construct high-order Nystr\"om quadratures
to solve the BIE \eqref{bie} itself;
the resulting tool is called QBX \cite{qbx}.

Firstly, in section \ref{s:err} we analyze (in Theorems~\ref{t:convd} and
\ref{t:convs}) the evaluation
error for Laplace double-layer and single-layer potentials,
with the global periodic trapezoid rule on an analytic curve $\pO$ with
analytic data---we are
surprised not to find these results in the literature.
This is crucial in order to determine the
neighborhood of $\pO$ in which native evaluation is poor;
it is within this ``bad neighborhood''
that the new close-evaluation scheme is used.
We present the scheme in section \ref{s:surr}:
simply put, the idea is to interpolate $\tau$ to a
fixed finer set of roughly $4N$ nodes, from which
one computes (via the addition theorem)
the coefficients of {\em local expansions}
(i.e., Taylor expansions in the Laplace case,
Fourier--Bessel expansions in the Helmholtz),
around a set of expansion centers placed near (but not too near) $\pO$.
It is these ``surrogate'' local expansions that are then evaluated at
nearby desired target points $x$.
This is reminiscent of the method of Schwab--Wendland \cite{schwabsurf}
but with the major difference that expansion centers lie off of, rather than
on, $\pO$.
In section \ref{s:ON} we show how to combine the new scheme
with the FMM to achieve an overall $O(N)$ complexity for the evaluation
of $O(N)$ target points lying in the bad neighborhood,
and apply this to high-frequency scattering from a smooth but complicated
obstacle 100 wavelengths across.
Finally, we conclude and mention future directions in section \ref{s:conc}.


\section{Theory of Laplace layer potential evaluation error using the
global trapezoid rule}
\label{s:err}

Our goal in this section is to analyze rigorously
the native (i.e., $N$-node) evaluation error for
analytic
single- and double-layer potentials for the Laplace equation ($\om=0$),
on analytic curves.
An example plot of such error varying over an interior domain
is shown in Fig.~\ref{f:err}(a).

\subsection{Geometric preliminaries}
\label{s:geom}

We identify $\RR$ with $\C$, and let the simple analytic closed curve 
$\pO$ define either 
a bounded interior, or unbounded exterior, open domain $\Omega\subset\C$.
We need $Z:\R \to \C$ as an analytic $2\pi$-periodic
counter-clockwise parametrization of
$\pO$, i.e.\ $Z([0,2\pi)) = \pO$.
This means $Z(s) = z_1(s) + iz_2(s)$, with $z_1$ and $z_2$ real analytic and
$2\pi$-periodic,
and that $Z$ may be continued as an analytic function
in some neighborhood of the real axis.
We assume that the {\em speed function} $|Z'(s)|$ is positive for all real $s$.
These conditions means that $Z$ is analytic and invertible
in a strip $|\im s|<\al$, for some $\al>0$.
The image of this strip under $Z$ defines an annular (tubular)
neighborhood of $\pO$ in which $Z^{-1}$ is also analytic.
Fig.~\ref{f:geom} shows such a strip and its image, both shaded in grey.
Also shown are the singularities that control its width:
at these points $Z'=0$, so that
locally $Z^{-1}$
takes the form of a (translated)
square-root map, hence must have branch cut
and cease to be single-valued.
These points are also singularities of the so-called {\em Schwarz function} of
the domain; see \cite[Ch.~5, 6, 8]{Da74} and \cite{bookrevshapiro}.
Other types of Schwarz singularities are possible
for domains with analytic boundaries (e.g.\ see \cite{mfs});
however, for our analysis the type of singularity is irrelevant.

For $\al\in\R$, we
will use the notation $\Gamma_\al$ to mean a translation of $\pO$
by $\al$ in the imaginary parameter direction,
$$
\Gamma_\al \;:=\; Z(\{s = t + i\al: t \in \R\})
~.
$$
In particular, $\Gamma_0 = \pO$.
Note that for all sufficiently small $\al$, $\Gamma_\al$ is a Jordan
curve, but that for larger $|\al|$, it will in general start to self-intersect.
This is illustrated by the images of the grid-lines in Fig.~\ref{f:geom}.
For $\al>0$ we define $A_\al$ by the open annular neighborhood
of $\pO$,
$$
A_\al \; := \; Z(\{s = t + ia: t,a \in \R, |a| < \al\})
~,
$$
i.e.\ the image of the strip $|\im s|<\al$.
Note that when $\Gamma_{-\al}$ and $\Gamma_\al$
do not intersect themselves or each other,
then $A_\al$ is simply the open region lying between them.

\bfi 
\hspace{.2in}
\ig{width=5in}{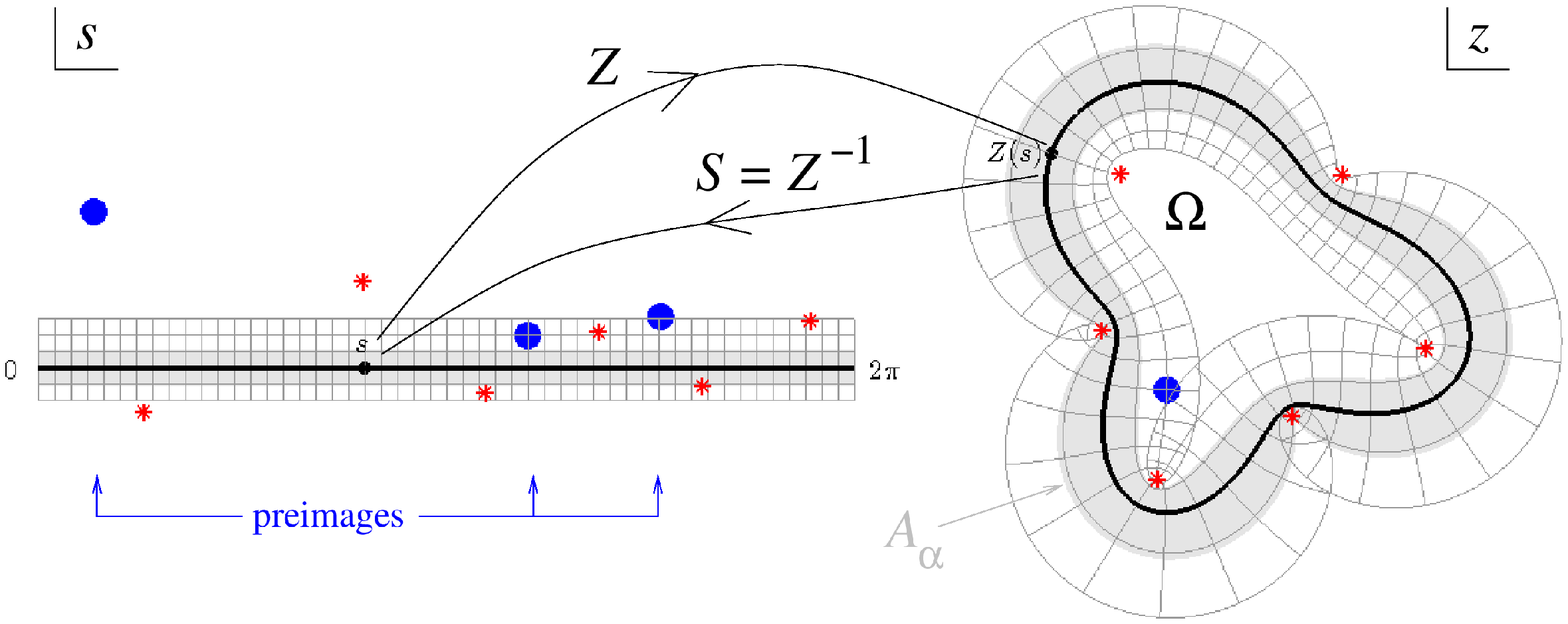}
\vspace{-2ex}
\ca{Parametrization of the boundary of the interior domain $\Omega$
from Fig.~\ref{f:err}, as a map
from $s$-plane (shown left with a square grid) to $z$-plane (shown right).
One point on the real $s$ axis and its image on the boundary are shown by small
black dots.
A single point in the interior and its three pre-images in the
$s$-plane are shown by large blue dots.
An annular neighborhood $A_\al$ in which $Z$ is analytic and invertible,
and its preimage, are shown in grey.
The nearest six (branch-type)
singularities of the Schwarz function of the domain,
and their pre-images, are shown by $\ast$.
}{f:geom}
\efi

\subsection{Evaluation error in the double-layer case}

Given a real-valued analytic density $\tau\in C(\pO)$,
the Laplace double-layer potential \eqref{dlp}
may be written $u = \re v$, where $v$ is the function defined by
the complex contour integral%
\footnote{Note that this Cauchy integral is not an example of
Cauchy's theorem, because $\tau$ is not the boundary value of
$v$. Rather, $\tau$ is purely real-valued.}
\be
v(z) \;=\; \frac{-1}{2\pi i} \int_\pO \frac{\tau(y)}{y-z} dy
~, \qquad z\in\CpO ~.
\label{vyd}
\ee
Thus $v$ is analytic in $\Omega$, and also in $\CO$.
Since the evaluation error of $u$ is bounded by that of $v$, we shall
work with $v$ from now on.
Let $\ttau$ be the pullback of $\tau$ under $Z$, i.e.\
$\ttau(s) = \tau(Z(s))$ for all $s\in\R$.
Rewriting \eqref{vyd} in terms of the parameter gives,
\be
v(z) \;=\;
\frac{-1}{2\pi i} \int_0^{2\pi} \frac{\ttau(s)}{Z(s) - z} Z'(s) ds ~,
\qquad z\in\CpO~.
\label{vd}
\ee
For quadrature of \eqref{vd}
we now choose the global periodic trapezoid rule along the real $s$ axis,
introducing nodes $2\pi j/N$,
$j=1,\dots, N$,
and equal weights $2\pi/N$, thus
\be
v^{(N)}(z) \;:=\;
\frac{-1}{iN} \sum_{j=1}^N \frac{\ttau(2\pi j/N)}{Z(2\pi j/N) - z}
Z'(2\pi j /N)~.
\label{vdN}
\ee



The integrand in \eqref{vd} is analytic, implying exponential convergence
of \eqref{vdN}
by the following classical
theorem \cite{davis59} (see e.g.\ \cite[Thm.~12.6]{LIE}).
\begin{thm}[Davis] 
Let $f$ be $2\pi$-periodic and analytic in the
strip $|\im s|\le \al$ for some $\al>0$, and let $|f|\le F$ in this strip.
Then the quadrature error of the periodic trapezoid rule,
\be
E_N \;:=\; \frac{2\pi}{N}\sum_{j=1}^N f(2\pi j/N) \; - \int_{0}^{2\pi} \! f(s) ds~,
\label{EN}
\ee
obeys the bound
\be
|E_N| \;\le\; \frac{4\pi F}{e^{\al N} - 1}
~.
\label{ENd}
\ee
\label{t:davis}
\end{thm} 

However, to achieve (rather than merely approach) the correct
convergence rate, we will need the following generalization
(similar to that of Hunter \cite{hunter71}):
\begin{lem}  
Let $f$ be $2\pi$-periodic and
meromorphic in the strip $|\im s|\le \al$ for some $\al>0$,
with only one simple pole in this strip, at $s_0$, with $\im s_0 \neq 0$.
Let $f$ have residue $r_0$ at this pole, and
let $|f|\le F$ on the edges of the strip, i.e. for all $s$ with
$|\im s|=\al$.
Then the quadrature error \eqref{EN}
obeys the bound
\be
|E_N| \;\le\; \frac{2\pi |r_0|}{e^{|\im s_0|N}-1} + \frac{4\pi F}{e^{\al N} - 1}
~.
\label{ENp}
\ee
\label{l:davispole}
\end{lem}   
Note that the first term dominates as $N$ grows,
and that $r_0=0$ recovers the Davis theorem.
\begin{proof}
Let $\Gamma_1$ and $\Gamma_2$ be the upper and lower strip boundaries
respectively, both traversed with increasing real part.
For the sum in \eqref{EN}
we apply the residue theorem to $\cot\frac{Ns}{2} f(s)$
in the strip, noticing that the vertical sides
cancel due to periodicity.
For the integral in \eqref{EN}
we apply the residue theorem to $f(s)$ in each of
the upper and lower semi-strips, take their average.
Combining these, \eqref{EN} can be rewritten
$$
E_N = \int_{\Gamma_1}
\left(\frac{i}{2}\cot \frac{Ns}{2} - \frac{1}{2}\right) f(s) ds -
\int_{\Gamma_2}
\left(\frac{i}{2}\cot \frac{Ns}{2} + \frac{1}{2}\right) f(s) ds
+ 2\pi i r_0 \left(\frac{i}{2}\cot\frac{Ns_0}{2} \mp \frac{1}{2}\right)
~,
$$
with the choice of sign in the last term corresponding to
the cases where $\im s_0$ has sign $\pm$.
The first bracketed term is bounded in size by $(e^{\al N}-1)^{-1}$
since $\im s = \al$. The same is true for the second bracketed term
since $\im s = -\al$.
The third bracketed term is bounded by $(e^{|\im s_0| N}-1)^{-1}$.
Combining these estimates and the bound on $f$ on the strip boundary
completes the proof.
\end{proof}

Now, since $\tau$ is real analytic on $\pO$,
it may be continued as a bounded holomorphic function in
the closure of some annular neighborhood $A_\al$.
Let us choose $\al>0$ so that $Z$ is also analytic and invertible
in the closure of $A_\al$ as discussed in section~\ref{s:geom}.
This is sufficient for
the pullback $\ttau$ to be bounded and holomorphic in the closed $s$-plane
strip of half-width $\al$.
%
Consider a target evaluation point $z \in A_\al$,
which then has a unique preimage $s=Z^{-1}(z)$ with $|\im s|<\al$.
Recalling the native evaluation \eqref{vdN}, define the error function
\be
\eps_N := \re (\vN - v)
~.
\label{ENfun}
\ee
Applying Lemma~\ref{l:davispole} in the strip $|\im s|<\al$,
noticing that the residue of the integrand in \eqref{vd} is just $\ttau(s)$,
and bounding the dominant first term in \eqref{ENp}
by a simple exponential, we have shown:
\begin{thm}  
Let $Z$ be the conformal map and $\tau$ the density
function defined at the beginning of this section.
Let $A_\al$, $\al>0$, be an
annular neighborhood in the closure of which $\tau$ is holomorphic and
bounded, and $Z^{-1}$ is holomorphic.
Then at each target point $z\in A_\al\backslash\pO$
we have exponential convergence of the error of the
Laplace double-layer potential evaluated with the $N$-point trapezoid rule
in the $s$ variable.
That is, there exist constants $C$ and $N_0$ such that
\be
|\eps_N(z)| \; \le \; C e^{-|\im s| N}
\qquad \mbox{ for all $N\ge N_0$ }
\label{eNd}
\ee
where $Z(s) = z$.
The constant $C$
may be chosen to be any number greater than
$|\ttau(s)| = |\tau(z)|$.
\label{t:convd}
\end{thm}   
%
To summarize:
convergence is exponential with rate given by the imaginary part of
the preimage of the target point under the complexification of the
boundary parametrization.

\begin{rmk}
It is possible to choose constants $C$ and $M$ for
which \eqref{eNd} holds uniformly in any {\em compact} subset of
$A_\al\backslash\pO$, but this is impossible over the
entire set $A_\al\backslash\pO$ because the value of $N$ at which
exponential convergence sets in diverges as $1/|\im s|$
as one approaches $\pO$.
Intuitively, this failure occurs because, no matter how large $N$ is,
individual quadrature points are always ``visible'' from close enough to
the boundary.
However, the constant $C$ may be chosen uniformly on the entire set to be
any number greater than $\sup_{z\in A_\al}|\tau(z)|$.
\end{rmk}

In Fig.~\ref{f:err} (a) we plot contours of constant $|\eps_N(z)|$
as the target point $z$ is varied over
a nonsymmetric interior domain $\Omega$ with analytic boundary,
for fixed $N$, and the simplest case $\tau\equiv1$ which
generates the potential $u \equiv -1$ in $\Omega$.
We overlay (as darker curves) predicted contours using Theorem~\ref{t:convd}.
In the annular neighborhood $A_\al$ (shown in grey in Fig.~\ref{f:geom})
the predicted contour for an error level $\eps$ is the curve
$\Gamma_{-\log(\eps/C)/N}$, where the lower bound $C = 1$ was used.
The match between the light and dark contours is almost perfect
(apart from periodic ``scalloping'' due to oscillation in the
error at the node frequency).

\begin{rmk}
Note that it would be possible to use the properties of the cotangent function
to improve Theorem~\ref{t:convd}
to include a {\em lower} bound on $|\vN(z)-v(z)|$
asymptotically approaching the upper bound. We have not pursued this, since
after taking the real part the error has no lower bound;
rather, it oscillates in sign at the node frequency
as shown by the ``fingers'' in Fig.~\ref{f:err} (a).
\end{rmk}

What happens further into the domain $\Omega$, i.e.\
for $\al$ values larger than that for which $Z$ is invertible?
Here, since $Z^{-1}$ starts to become multi-valued,
there are multiple preimages 
which lie within a given
strip $|\im s|\le \al$ (e.g.\ see large dots in Fig.~\ref{f:geom}).
To analyze this would
require a variant of Lemma~\ref{l:davispole} with multiple poles.
We prefer an intuitive explanation.
Let us assume that $\ttau$ remains holomorphic throughout such a wider
strip.
Then it is clear that
the $s$-plane pole {\em closest to the real axis} will dominate
the error for sufficiently large $N$
because it creates the slowest exponential decay rate.
Hence, to generate each predicted contour in Fig.~\ref{f:err}
we use the set of $z$-plane
points which have their closest preimage a distance
$\al = -\log(\eps/C)/N$ from the real axis.
For each $\al$, this curve is simply the boundary of $A_\al$,
%
%
i.e.\
the self-intersecting curve $\Gamma_\al$ with all its ``loops trimmed off.''
We see in Fig.~\ref{f:err}(a) that, throughout the interior of $\Omega$,
this leads to excellent prediction of the error
down to at least 14 digits of accuracy.
Even features such as the cusps which occur beyond the two
closest interior Schwarz singularities (at roughly 4 o'clock and 7 o'clock) are
as predicted.

Eventually, for a target point deep inside the domain, all of its preimages
may be further from the real axis than the widest strip in which
$\ttau$ is holomorphic. In this case, one is able to apply
only
the Davis theorem, and the width of the strip in which $\ttau$ is holomorphic
will now control the error (this case is never reached in Fig.~\ref{f:err}(a),
although it will be in (d)).

We now perform some instructive variants on this numerical experiment.
In Fig.~\ref{f:err}(b) we use the same domain and $N$ as in (a),
but instead of using a given $\tau$, we solve for $\tau$ via
\eqref{bie} with the $N$-point Nystr\"om method,
given (entire) Dirichlet data $u(x,y) = xy$.
This is a typical BVP setting, albeit a simple one.
We see that the errors are similar to (a) with the major difference that
the errors bottom out at around $10^{-9}$: this is because $\tau$ itself
only has this accuracy for the $N=60$ nodes used
($N\ge 130$ recovers full machine precision in $\tau$).

In (c) we repeat (a) except using a boundary shape $\pO$ distorted by
a localized Gaussian ``bump'' at around 11 o'clock. The
errors are now never smaller
than $10^{-8}$: note that since $\tau\equiv 1$ this
cannot be due to inaccuracy, nor to lack of sufficient analyticity, in $\tau$.
Rather, the mechanism is the rapid 
growth in distance from $\pO$ of
the contours $\Gamma_\al$ in this region, as $\al$ increases.
This is verified by the quite good agreement with predicted contours.
We observe that such growth is typical in a region with rapidly-changing
curvature, which explains the well-known empirical rule that,
for high accuracy,
$N$ must be chosen large enough to resolve such spatial features.

To remind the reader that Theorem~\ref{t:convd} predicts errors just as well
in the exterior as in the interior, we suggest a glance at
Fig.~\ref{f:err}(d), to be discussed more later.
We conclude with a remark about the universality of the
``safe'' distance 
from the boundary for accurate evaluation.

\begin{rmk}[``$5h$ rule'']
In practical settings, if
the evaluation point is a distance $5h$ or more from the boundary,
where $h$ is the local spacing between the nodes $Z(2\pi j/N)$,
then around 14 digits of accuracy in $u^{(N)}$ is typical.
This is because, when the local distortion induced by the
conformal map $Z$ is small, the preimage is then a distance roughly
$5\cdot 2\pi/N$ from the real axis,
giving the term $e^{-5\cdot 2\pi} \approx 2\times 10^{-14}$ in \eqref{eNd}.
This relies on two assumptions:
i) $\tau$ is analytic and bounded in an annular
neighborhood of sufficient width,
and ii) the local distortion is small on a spatial scale of a few times $h$.
Why should these hold in practice?
The answer is that they are preconditions for the Nystr\"om method
to produce a highly accurate solution density $\tau$ in the first place.
\label{r:5h}
\end{rmk} 

Finally, we note that if the periodic trapezoid rule
were to be replaced by a panel-based quadrature formula with Chebyshev
node density, such as Gauss--Legendre,
a similar analysis to the above would show that the contours
of error level are the images under $Z$ of the Bernstein
ellipses \cite{davisrabin,ATAP} for the panel intervals.

\subsection{Single-layer case}
\label{s:slp}

\bfi 
\centering\ig{width=5in}{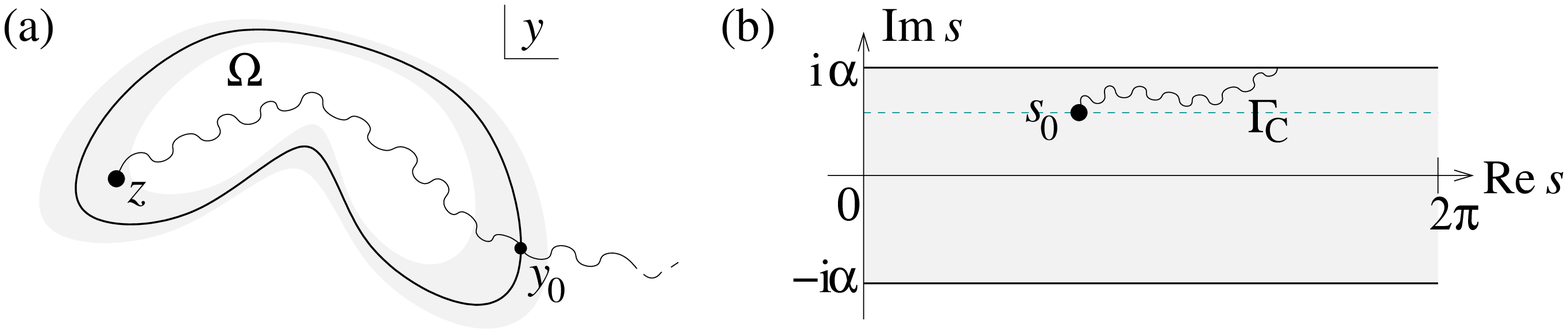}
\ca{(a) A branch cut in the complex $y$ plane
for the function $\log 1/(y-z)$ in \eqref{vs} for $z\in\Omega$, with
exit point $y_0\in\pO$ fixed; see Remark~\ref{r:branch}.
(b) The analytic strip (excluding branch cut $\Gamma_C$) for the integrand
in Lemma~\ref{l:davisbranch}.
}{f:branchcuts}
\efi

We now present a similar analysis for single-layer potential
evaluation, including a numerical verification.
Recalling the fundamental solution \eqref{Phi}, the single-layer potential
is
\be
u(x) = \int_\pO \Phi(x,y) \sigma(y) ds_y ~,  \qquad x\in\CpO~.
\label{slp}
\ee
For the proof in the Laplace case ($\om=0$), we need the analytic
function of which this is the real part.
A real-valued analytic density $\sigma\in C(\pO)$
generates a potential $u=\re v$, where, analogously to \eqref{vyd} and
\eqref{vd},
\be
v(z) \;=\;
\frac{1}{2\pi} \int_\pO \biggl(\log\frac{1}{y-z}\biggr)
\sigma(y) \,|dy|
\;=\;
\frac{1}{2\pi} \int_0^{2\pi} \biggl(\log\frac{1}{Z(s) - z} \biggr)
\tsig(s) \,|Z'(s)|\, ds ~,
\quad z\in\CpO~,
\label{vs}
\ee
and $\tsig(s) := \sigma(Z(s))$ is the pullback.
Note that now the magnitude of $dy$ rather than its complex value is taken.
The multiple sheets of the imaginary part of the logarithm cause the following complication.

\begin{rmk} 
Without a careful choice of the branch cuts of the kernel
$L(y,z) := \log 1/(y-z)$ in \eqref{vs},
$\im v$ would fail to be a harmonic conjugate of $u$, and
$v$ would not be holomorphic.
However, it is easy to check that
sufficient conditions are as follows, which we will from now assume
apply in the definition \eqref{vs}.
For the case $z\in\Omega$, the branch cut of $L(\cdot,z)$
must exit $\Omega$ only at a single point $y_0\in\pO$ independent of $z$
(see Fig.~\ref{f:branchcuts}(a)),
and, for each fixed $y\in\pO$, $y\ne y_0$, $L(y,\cdot)$ is continuous
in $\Omega$.
For the case $z\in\C\backslash\overline{\Omega}$,
the branch cut in $L(\cdot,z)$ must avoid $\pO$,
and, for each $y\in\pO$, the branch cut in $L(y,\cdot)$ passing to infinity
must avoid $\Omega$.
\footnote{Note that, unless $\int_\pO \sigma(y) |dy| = 0$,
i.e.\ total charge vanishes,
then $\im v$ itself cannot be single-valued outside $\Omega$.}
\label{r:branch}\end{rmk} 

As in the double-layer case, we will assume that
$\tsig$ continues to a function analytic in some
strip $|\im s|\le \al$ in which $Z$ is analytic and invertible.
The following is a variation on Lemma~\ref{l:davispole}.

\begin{lem}  
Let $f$ be $2\pi$-periodic and analytic
everywhere in the strip $|\im s|\le \al$ apart from on a
branch cut $\Gamma_C$ which starts from the point $s_0$, with $\im s_0 \neq 0$,
then proceeds to the nearer edge of the strip
while avoiding the region $|\im s|\le|\im s_0|$ (see Figure~\ref{f:branchcuts}(b)).
Let $|f|\le F$ on the edges of the strip.
Then the quadrature error \eqref{EN}
obeys the bound
\be
|E_N| \;\le\;
\frac{1}{e^{|\im s_0|N}-1} \int_{\Gamma_C} |f^+(s)-f^-(s)|\, |ds| \; +\;
\frac{4\pi F}{e^{\al N} - 1}
~,
\label{ENb}
\ee
where $f^+$ and $f^-$ are the limiting values on
either side of the branch cut.
\label{l:davisbranch}
\end{lem}   
\begin{proof}
As in the proof of Lemma~\ref{l:davispole}, we use the residue 
and Cauchy theorems to write $E_N$ as
$$
\int_{\Gamma_1}
\left(\frac{i}{2}\cot \frac{Ns}{2} - \frac{1}{2}\right) f(s) ds -
\int_{\Gamma_2}
\left(\frac{i}{2}\cot \frac{Ns}{2} + \frac{1}{2}\right) f(s) ds
+ \int_{\Gamma_C} \left(\frac{i}{2}\cot\frac{Ns}{2} \mp \frac{1}{2}\right)
[f^+(s)-f^-(s)]ds
~,
$$
taking care to include the traversal of both sides of $\Gamma_C$ in
the relevant contours.
As before, the first two integrals may be bounded to give the
second term in \eqref{ENb}.
The last integral may be bounded by the first term in \eqref{ENb}.
\end{proof}

We now apply this estimate to the native evaluation error for
the single-layer potential \eqref{slp}
and find exponential convergence with the same rate as for the double-layer.

\begin{thm}  
Let $A_\al$, $\al>0$, be an 
annular neighborhood in the closure of which $\sigma$ is holomorphic and
bounded, and $Z^{-1}$ is holomorphic.
Let $\eps_N$ be the evaluation error function of the Laplace
single-layer potential (the real part of \eqref{vs})
with the $N$-point periodic trapezoid rule in the $s$ variable.
Then at each target point $z\in A_\al\backslash\pO$, there exist constants
$C$ and $N_0$ such that
$$
|\eps_N(z)| \; \le \; C e^{-|\im s| N}
\qquad \mbox{ for all } N\ge N_0
$$
where $Z(s)=z$.
\label{t:convs}
\end{thm}   
\begin{proof}
We use the notation $\overline{s}$ to indicate the complex conjugate of
$s$.
We note that the analytic continuation of $|Z'(s)|$ off the real axis
is $\bigl(Z'(s)\overline{Z'(\overline{s})}\bigr)^{1/2}$, which is analytic
and bounded in the strip.
We also note that the imaginary part of \eqref{vs} fails to be $2\pi$-periodic,
due to the jump in the logarithm by $2\pi i$ at $y=y_0$.
We choose $y_0=Z(0)$, and further restrict the branch cut of $\log 1/(y-z)$
in \eqref{vs}
to have a preimage which crosses the strip only along the imaginary $s$ axis.
Then we may then apply the function
$$f(s) = \frac{1}{2\pi} \left(\log \frac{1}{Z(s)-z} + is \right) \tsig(s) \bigl(Z'(s)\overline{Z'(\overline{s})}\bigr)^{1/2}$$
in Lemma~\ref{l:davisbranch}.
Here the new term $is$ cancels the imaginary jump in the logarithm,
making $f$ periodic,
yet does not affect the real part of the integral.
Since the jump in the logarithm everywhere on $\Gamma_C$ is $2\pi i$, then
$\int_{\Gamma_C} |f^+(s)-f^-(s)|\,|ds|$ is bounded by a constant
involving the lengths of $Z(\Gamma_C)$ and $Z(\overline{\Gamma_C})$
and an upper bound on $|\tilde\sigma|$ in the strip.
Analogously to the proof of Theorem~\ref{t:convd},
this bounds the first exponential term in \eqref{ENb};
the weaker second term can be absorbed into the constant.
\end{proof}

We now test the convergence for a combination of single- and double-layer
densities, by checking the accuracy of the
following Green's representation formula (GRF).
For $u$ a solution to the PDE \eqref{bvp1} in a
bounded domain $\Omega\subset\RR$,
\be
\int_\pO \Phi(x,y) \frac{\partial u}{\partial n}(y) \,ds_y
- \int_\pO \frac{\partial \Phi(x,y)}{\partial n(y)} u(y) \,ds_y
\;= \;
\left\{\begin{array}{ll} u(x), & x\in\Omega,\\
0, & x\in\RR\backslash\overline{\Omega}.\end{array}\right.
\label{grf}
\ee
Errors for the native evaluation of this in both interior and exterior cases
are shown together in
Fig.~\ref{f:err}(d): it is clear that the errors grow exponentially
as $x$ approaches $\pO$, from either inside or outside.
In fact, this figure shows the case of the low-frequency
Helmholtz equation with $\om=2$; the Laplace plot is almost identical.
This confirms the intuition that the convergence for
the low-frequency Helmholtz equation is very similar to that for the
Laplace equation.
(See section~\ref{s:helm} for the formulae for the Helmholtz case.)

Why do the the errors stop decreasing with distance once an error of
$10^{-12}$ is reached in Fig.~\ref{f:err}(d)?
This cannot be due to inaccurate boundary data, since the data is exact
to machine precision.
Rather, the answer lies in the fact that the
boundary data derives from a point source
(shown by a large dot) outside, but not too far from, $\Omega$.
This means that the densities cannot be analytic in a larger
annular neighborhood, limiting the maximum convergence rate in $N$ to the
value at this singularity location. 
This illustrates the effect of densities which are not entire functions.

\bfi 
\bmp{3in}(a)
{\small $\log_{10} |u^{(N)}-u|/\|u\|_\infty$, DLP, $u(x,y) = e^y\cos x$}\\
\ig{width=2.8in}{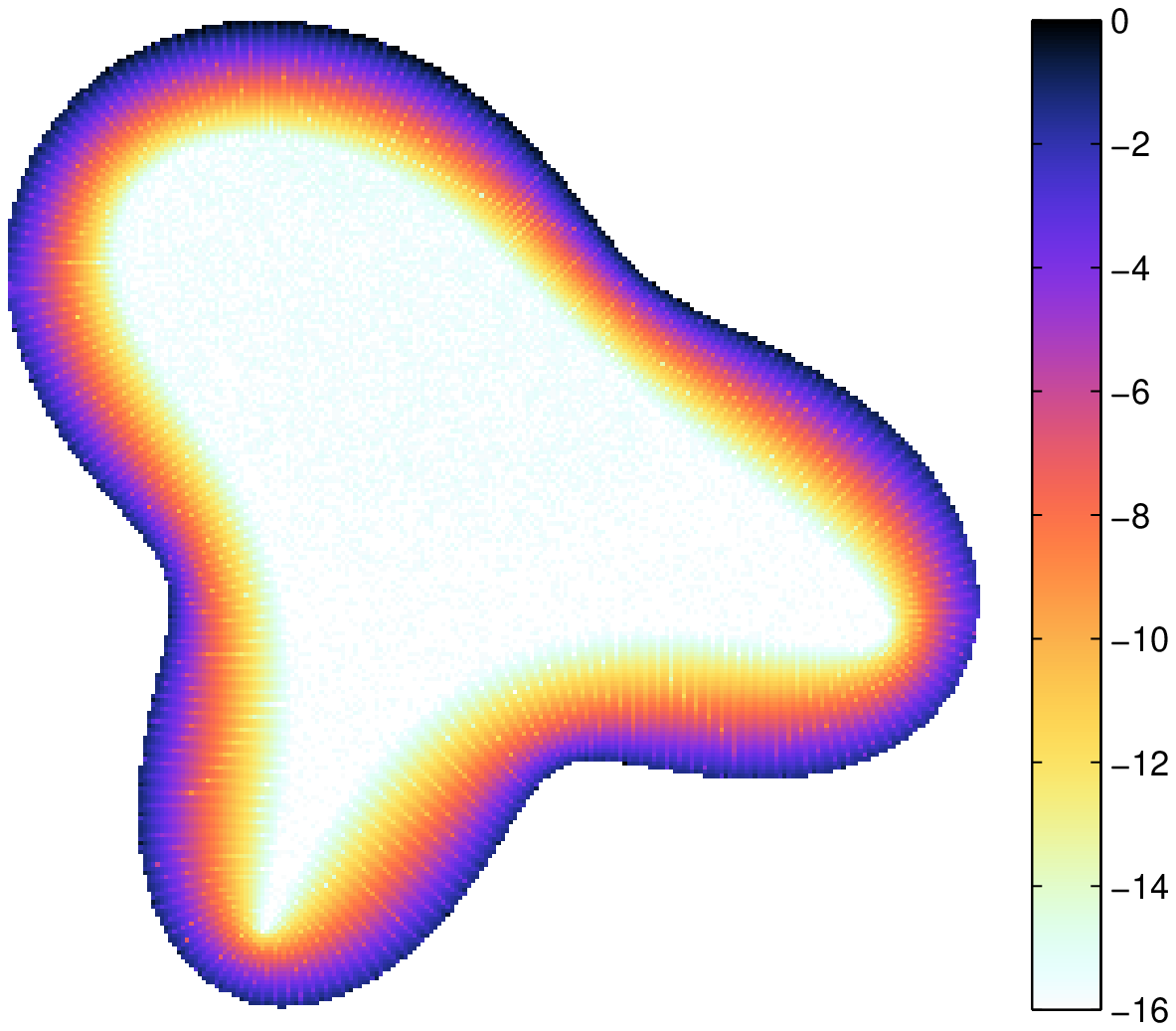}
\emp
\bmp{3in}(b)
{\small $\log_{10} |\hat u-u|/\|u\|_\infty$, DLP, $u(x,y) = e^y\cos x$}\\
\ig{width=2.8in}{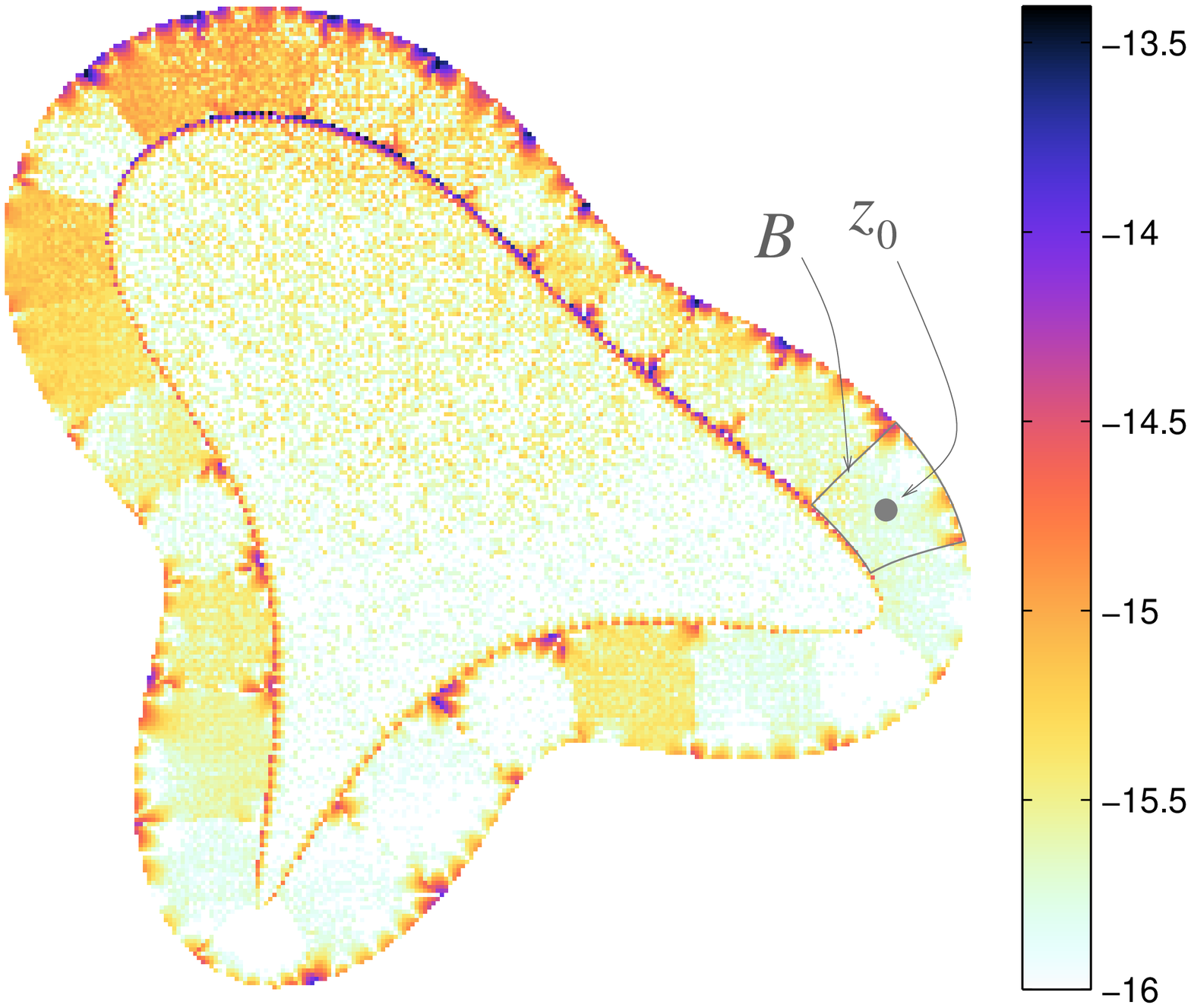}
\emp
\ca{Error (relative to largest value)
in evaluation of the solution of the Laplace BVP with
Dirichlet data $u(x,y) = e^y \cos x$, with $N=130$,
using the native quadrature (a),
and the proposed scheme (b) with $N_B=26$, $p=10$ and $\Nf=4N$.
Note the change in color scale: the relative $L^\infty$ error is $4\times 10^{-14}$,
and the relative $L^2$ error $4\times 10^{-15}$.
One box $B$ and its center $z_0$ are labeled; many other boxes and the
the annular half-neighborhood
$\Omega_\tbox{bad} = A_\tbox{bad} \cap \Omega$ are visible due to jumps
in the error.
}{f:rfnk0}
\efi

\section{Close evaluation of Laplace layer
potentials by surrogate local expansions}
\label{s:surr}

In this section we present and analyze the main new scheme.
We have seen that the native evaluation error is exponentially small far from
$\pO$, but unacceptably large close to $\pO$.
Thus we define the ``bad annular neighborhood'' $A_\tbox{bad} := A_{\al_\tbox{bad}}$
where, following Remark~\ref{r:5h}, we will choose $\al_\tbox{bad} = 10\pi/N$,
that is, a distance around $5h$ either side of $\pO$,
giving around 14 digit expected accuracy throughout
$\RR \backslash A_\tbox{bad}$.
(Obviously $\al_\tbox{bad}$ could be decreased if less accuracy is desired.)
We now describe the new method for layer potential
evaluation in this bad annular neighborhood,
focusing on the case of $\Omega_\tbox{bad} := A_\tbox{bad} \cap \Omega$,
for $\Omega$ an interior domain with
$\pO$ traversed in the counter-clockwise sense.
Thus, target preimages will have positive imaginary part.
(The exterior case is analogous.)

Let us fix $N$ and assume that a double-layer potential $\tau$ has been
approximately computed by the \Ny\ method,
so is represented by $\{\tau_j\}_{j=1}^N$, its values at the nodes.
A precondition for accuracy of the \Ny\ method
is that $\tau$ is well-approximated
by its \Ny\ interpolant through these nodes.
Recalling that $u = \re v$, we wish to evaluate $v$ via \eqref{vyd}
in $\Omega_\tbox{bad}$.
We cover $\Omega_\tbox{bad}$ by non-intersecting ``boxes''
$B_b$, $b=1,\ldots, N_B$. The $b$th box $B_b$
is the image under $Z$ of the
$s$-plane rectangle
$2\pi (b-1/2)/N_B \le \re s < 2\pi (b+1/2)/N_B$,
$0\le \im s \le \al_\tbox{bad}$.
By choosing $N_B = \lceil N/5 \rceil$ the boxes become very close
to being square.
Fig.~\ref{f:rfnk0}(b) shows the annular half-neighborhood and one box.

Let $B = B_b$ be the $b$th box. We choose an expansion center
\be
z_0 := Z(2\pi b/N_B + i\al_0)
\label{z0}
\ee
where the choice of imaginary distance $\al_0 = \al_\tbox{bad}/2$
places $z_0$ roughly central to $B$, and roughly $2.5\,h$
from the boundary.
We represent $v$ by a 
Taylor series
$$
v(z) = \sum_{m=0}^{\infty} c_m (z-z_0)^m
~,
$$
which converges uniformly in $B$ if $v$ is analytic in some disc
centered at $z_0$ with radius greater than $R$,
where $R$ is the maximum radius of the box,
$$
R := \sup_{z\in B}|z-z_0|
~.
$$
Each coefficient $c_m$ can be computed, using the Cauchy formula for
derivatives, as
\be
c_m \;=\; \frac{-1}{2\pi i} \int_\pO \frac{\tau(y)}{(y-z_0)^{m+1}} dy
 = \frac{-1}{2\pi i} \int_0^{2\pi} \frac{\ttau(s)}{(Z(s)-z_0)^{m+1}} Z'(s) ds
~, \qquad m=0,1,\ldots
\label{cmd}
\ee
To approximate the latter integral we use the periodic trapezoid rule
with $\Nf$ new nodes, $s_j = 2\pi j/\Nf$, $j=1,\ldots,\Nf$,
which we call ``fine'' nodes, that is,
\be
c_m \;\approx\;
\hat{c}_m \;:=\;
\frac{-1}{i \Nf} \sum_{j=1}^{\Nf} \frac{\ttau(s_j)}{(Z(s_j)-z_0)^{m+1}} Z'(s_j)
~.
\label{cmhd}
\ee
Since $z_0$ is in the bad annular neighborhood, clearly we need $\Nf>N$ to
evaluate even the coefficient $c_0$ accurately:
one would expect, since $\al_0 = \al_\tbox{bad}/2$, that $\Nf = 2N$ would
be sufficient. However, for
large $m$ the term $1/(Z(s)-z_0)^{m+1}$ is now {\em oscillatory},
so an even larger $\Nf$ will be needed.
Note that $\ttau(s_j)$ must be found by interpolation from
its values at the $N$ \Ny\ nodes.
Our surrogate potential in the box $B$ is then simply the
truncated series with the above approximated coefficients,
and $p$ terms, i.e.\
\be
v(z) \approx \hat{v}(z) = \sum_{m=0}^{p-1} \hat{c}_m (z-z_0)^m
~,
\qquad \mbox{$z$ in box $B$}
~.
\label{vqbx}
\ee
Finally the real part must be taken, thus in each of the
boxes we set $\hat{u} = \re \hat{v}$.
We claim that, if the parameters $p$ and $\Nf$ are well chosen,
then $\hat{u} \approx u$ uniformly
to high accuracy throughout $\Omega_\tbox{bad}$.
(In $\Omega\backslash\Omega_\tbox{bad}$
we revert to $\hat{u} = u^{(N)}$, the native evaluation scheme.)

This is 
illustrated in Fig.~\ref{f:rfnk0}:
we first solve \eqref{bie} for $\tau$, given entire boundary data,
using $N=130$, which is around the minimum $N$ required
to achieve full accuracy in $\tau$.
We then try the new evaluation scheme,
using a standard trigonometric polynomial interpolant
(e.g.\ see \cite[Sec.~11.3]{LIE}, implemented easily via the FFT)
to approximate $\tau$ at the fine nodes, which gives a $\hat{u}$ with
uniform relative accuracy of 13.5 digits (14 digits in the $L^2(\Omega)$
norm).
Note that, since the boundary data is entire, this is a well-behaved
example; we consider more challenging cases after the following analysis.


\subsection{A convergence theorem for the Laplace double-layer potential}
\label{s:thm}

Because the boxes touch $\pO$,
uniform convergence of the Taylor series of $v$ in
any box generally demands that $v$ have an analytic continuation
some distance {\em outside} $\Omega$, which might seem far-fetched.
But it turns out that the region of analyticity of $v$ is 
at least as large as that of $\tau$.
\begin{pro} 
Let $\tau$ be analytic in some closed annular neighborhood $A$ of $\pO$.
Then $v$ given by \eqref{vd} is analytic in $\Omega$ and
moreover
continues 
as an analytic function throughout $A \backslash \Omega$.
\label{p:vanal}
\end{pro} 
\begin{proof}
Let $\Gamma$ be the 
exterior boundary of $A$.
Consider the identity arising from \eqref{vyd},
$$
\frac{-1}{2\pi i} \int_\Gamma \frac{\tau(y)}{y-z} dy
\;= \;
v(z) + \frac{-1}{2\pi i} \int_{\Gamma-\pO} \frac{\tau(y)}{y-z} dy
~,
$$
where $-\pO$ means $\pO$ traversed in the opposite direction.
Both sides equal $v(z)$ for $z\in\Omega$
since the second term on the right is zero by Cauchy's theorem.
However, as a function of $z$,
the left side is analytic throughout $A \cup \Omega$,
and thus provides the desired analytic continuation of $v$.
\end{proof}

To state a convergence result,
some control is needed of the distortion induced by the conformal
map $Z$. Let $A_\al$ be an annular neighborhood in which $Z^{-1}$ exists.
Recall that box $B$ has a center $z_0$ (which we assume is in $A_\al$)
with $\al_0$ the imaginary part of its preimage, and radius $R$.
Given this, we define a {\em geometric distortion quantity},
\be
\gamma = \gamma_{z_0,R} := \sup_{0<a<\al_0} \frac{R}{d(\Gamma_a,z_0)}\frac{\al_0-a}{\al_0}
~,
\label{gamma}
\ee
where $d(X,y)$ means the minimum Euclidean distance from the point $y$ to
points in the set $X$.
It is clear that, as $a>0$ approaches $\al_0$ from below,
the curve $\Gamma_a$ first touches $z_0$ at $a=\al_0$;
an interpretation of $\gamma^{-1}$ is a scaled lower bound on the ratio
between its distance and $\al_0-a$
(we note that in peculiar geometries it may be that the nearest part of
$\Gamma_a$ is not within the box $B$).
For an undistorted square box with $z_0$ at its center, $\gamma=\sqrt{2}$; in
practical settings $\gamma$ is around 1.5 to 2.
For the geometry in
Fig.~\ref{f:rfnk0}(b), the median $\gamma$ is 1.7 and the maximum 2.4
(occurring at around 7 o'clock, where the closest
interior Schwarz singularity lies).\footnote{These
$\gamma$ values are accurately estimated using
20 values of $\al$ with $10^3$ points on each curve $\Gamma_\al$.}

Our result concerns convergence of the surrogate scheme
simultaneously in the expansion order $p$ and the number of fine nodes $\Nf$.

\begin{thm} 
Let $B$ be a box with radius $R$.
Let the center $z_0$ have $\im z_0 = \al_0$,
with $A_{\al_0}$ an annular neighborhood in which $Z^{-1}$ is holomorphic,
and in which the 
density 
$\tau$ is holomorphic and bounded.
Let the double-layer potential
$v$ given by \eqref{vyd} be analytic in an open neighborhood of some closed
disc of radius $\rho>R$ about $z_0$.
Let $\hat v$ be given by the surrogate scheme
\eqref{vqbx} with $\hat{c}_m$ given
by \eqref{cmhd}, using the exact density at the fine nodes
$\ttau(s_j) = \tau(Z(s_j))$, $j=1,\dots,\Nf$.
Then the error function
$$
\hat\eps = \re( \hat {v} - v)
$$
has the uniform bound
\be
\hat{\eps}_B \;:=\; \sup_{z\in B} |\hat\eps(z)|
\; \le \;
C \left(\frac{R}{\rho}\right)^p +
C (e\gamma \al_0 \Nf)^p p^{1-p} e^{-\al_0 \Nf}
\qquad \mbox{ for all } \Nf \ge 1/\al_0, \; 1\le p\le \Nf/2
~,
\label{eps}
\ee
where $C$ indicates constants that depend on $\pO$, $Z$, $\tau$, $B$ and $z_0$,
but not on $p$ nor $\Nf$.
\label{t:eps}
\end{thm} 

The restriction to $\Nf \ge 1/\al_0$ always holds in practice
because $N$ is already many times larger than $1/\al_0$, whilst $\Nf>N$.
The restriction on $p$ also holds in practice, since usually $\Nf>10^2$.
The first term in \eqref{eps} is simply the truncation
error of the Taylor series for $v$ in the box, so is independent
of $\Nf$.
Note that the required analyticity of $v$ is already given
by Prop.~\ref{p:vanal} and the analyticity of $\tau$,
with $\rho \approx \sqrt{2} R$,
unless the local distortion is very large.

The $\Nf$ dependence of the second term 
may be written $e^{-\al_0 \Nf + p \log \Nf}$, showing that (at fixed
$p$) this term is asymptotically exponentially convergent in $\Nf$
with rate $\al_0$.
Thus the whole scheme is also asymptotically
exponentially convergent: given arbitrary
$\eps>0$, by fixing $p$ such that the first term is smaller than
$\eps/2$, one may find a value of $\Nf$ such that the second term is
also smaller than $\eps/2$, and have $\hat\eps_B\le \eps$.
There is a subtlety. Fixing $\Nf$ while increasing
$p$ is a {\em bad idea}:
in practice it leads to exponential {\em divergence},
because the factor $e\gamma \al_0 \Nf$ is usually large (around 300,
independent of the problem size),
and hence the second term grows exponentially in $p$ for typical
$p$ (less than 30).

\begin{proof} 
We will use $C$ to indicate (different) constants that have
only the dependence stated in the theorem.
Comparing $\hat{v}$ to the exact Taylor series for $v$,
and using $|z-z_0|\le R$,
\be
\hat\eps_B \;\le\;
\sup_{z\in B}\biggl|\sum_{m=0}^{p-1} \hat{c}_m (z-z_0)^m - v(z)\biggr|
\; \le \;
\sum_{m\ge p}|c_m| R^m +
\sum_{m=0}^{p-1}|\hat{c}_m-c_m| R^m
~.
\label{epsest}
\ee
For the first term we use $|c_m|\le C /\rho_m$, which follows from the
Cauchy integral formula in the disc of radius $\rho$
(e.g.\ \cite[Cor.~4.3]{steinshakarchi}), and bound the geometric
sum.
For the second term we apply
Theorem~\ref{t:davis} to the $s$-integral \eqref{cmd}
in a strip of width $\al$, so that, for each $\al \in (0,\al_0)$,
\be
|\hat{c}_m-c_m| \;\le\; 
\left(\sup_{\im s = \pm \al} |\ttau(s) Z'(s)| \right)
\frac{2}{d(\Gamma_\al,z_0)^{m+1}} \frac{1}{e^{\al \Nf} - 1}
\;\le\;
\frac{C}{d(\Gamma_\al,z_0)^{m+1}} \frac{1}{e^{\al \Nf} - 1}
~,
\label{cmerr}
\ee
where the second estimate follows from the
maximum modulus principle and boundedness of $Z'$ and $\ttau$ in the
$\al_0$-strip.
Inserting the above two estimates into \eqref{epsest},
then bounding each term in the sum by the last 
term since
$R>d(\Gamma_\al,z_0)$, gives
\bea
\hat\eps_B
& \le &
\frac{C}{1-R/\rho} \left(\frac{R}{\rho}\right)^p
+
\sum_{m=0}^{p-1} R^m \frac{C}{d(\Gamma_\al,z_0)^{m+1}} \frac{1}{e^{\al \Nf} - 1}
\; \le \;
C\left(\frac{R}{\rho}\right)^p
+
C p \biggl(\frac{R}{d(\Gamma_\al,z_0)}\biggr)^p
\frac{1}{e^{\al \Nf} - 1}
\nonumber \\
& \le &
C\left(\frac{R}{\rho}\right)^p
+
C p \biggl(\frac{\gamma \al_0}{\al_0-\al}\biggr)^p e^{-\al\Nf}
\qquad \mbox{ for all } \Nf > 1/2\al
~,
\label{epsbnd}
\eea
where in the last step we used \eqref{gamma} and simplified the exponential
bound.
This estimate holds for each $\al \in (0,\al_0)$;
however, if $\al$ is chosen too small, the last exponential will
decay very slowly. Conversely, if $\al$ approaches $\al_0$ then
the distance $d(\Gamma_\al,z_0)$ vanishes and its $p$th negative
power blows up rapidly.
For each $p$ and $M$, the optimal value $\hat\al$ which minimizes
this second term is found
by setting $\partial/\partial\al$ of the logarithm of this term to zero,
and solving, giving
$$
\hat\al = \al_0 - \frac{p}{\Nf}
$$
which by the condition on $p$ is never smaller than $\al_0/2$.
Substituting $\al=\hat\al$ in \eqref{epsbnd} gives \eqref{eps}.
\end{proof} 




To interpret the theorem we need to relate it to the original $N$
\Ny\ nodes (notice that $N$ does not appear in the theorem).
We introduce two dimensionless parameters. Let
$$\delta \;:=\; \frac{\al_0 N}{2\pi}$$
be the distance of the center $z_0$ from the boundary
in units of local node spacing $h$; recall that we set $\delta = 2.5$
in the above.
Let $$\beta \;:=\; \frac{\Nf}{N}$$
be the ``upsampling ratio'', the ratio of the number
of fine nodes to the number of original \Ny\ nodes.

It is clear that the effort to compute \eqref{cmhd} as presented
scales as $O(\beta)$, once $N$, $N_B$ and $p$ are fixed.
Thus we wish to know the minimum $\beta$ needed,
and rewrite the second term in \eqref{eps} as
\be
C p \exp \biggl(-2\pi \delta \beta + p \log
\bigl(2\pi \delta \beta \frac{\gamma e}{p}\bigr)\biggr)
~.
\label{2nd}
\ee
For instance, fixing $\delta=2.5$ and taking $\gamma \approx 1.7$,
we may solve (via rootfinding) for the approximate $\beta$ required for
convergence by assuming that $C$ is $O(1)$
then equating the exponential term in \eqref{2nd} to
the desired error level, e.g.\ $10^{-14}$.
The predicted results are: for $p=10$, $\beta = 4.2$ is sufficient.
This matches well the finding that $\beta = 4.0$ was
sufficient to give around 14-digit accuracy in the example
of Fig.~\ref{f:rfnk0}(b).
For $p=20$, $\beta = 5.9$ is sufficient,
indicating that $\beta$ need grow only weakly with $p$.

\begin{rmk}[box centers]
Equation \eqref{2nd} suggests that increasing $\delta$ (moving the
centers towards the far edge of their boxes) might increase
accuracy. In practice, however, we find that
this does not help because both $R/\rho$ and $\gamma$ increase
for the worst-case boxes.
\end{rmk}

\bfi 
\hspace{-.1in}
\bmp{2.1in}(a)
{\small $\log_{10} \|\hat u-u\|_\infty/\|u\|_\infty$, \, $N{=}130$\\
\mbox{}\qquad $u(x,y) = e^y\cos x$}\\
\ig{width=2in}{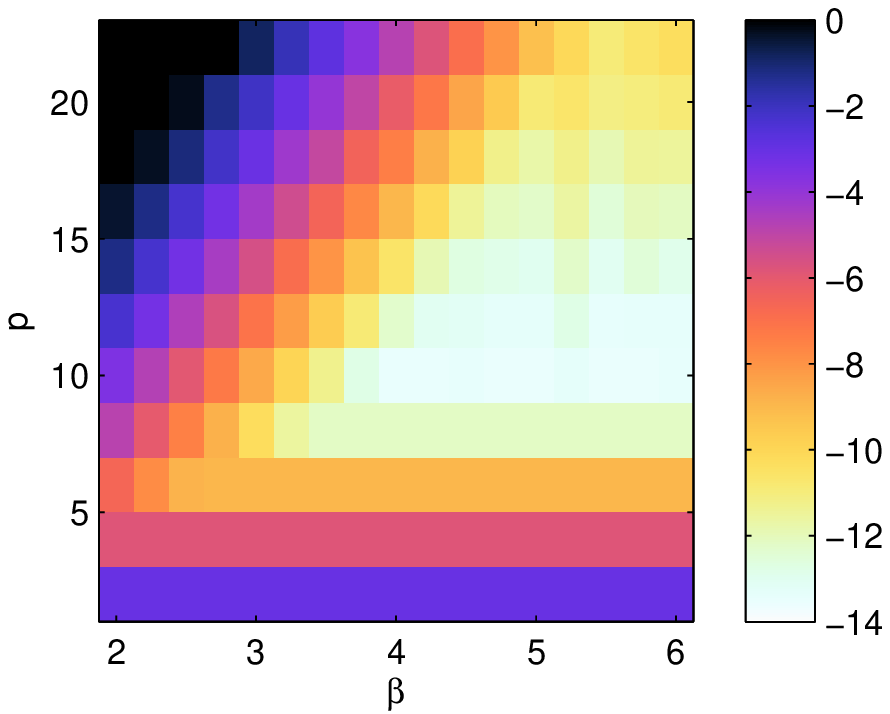}
\emp
\bmp{2.1in}(b)
{\small $\log_{10} \|\hat u-u\|_\infty/\|u\|_\infty$, \, $N{=}180$\\
\mbox{}\qquad $u$ pt src, \;\; \Ny\ interp}\\
\ig{width=2in}{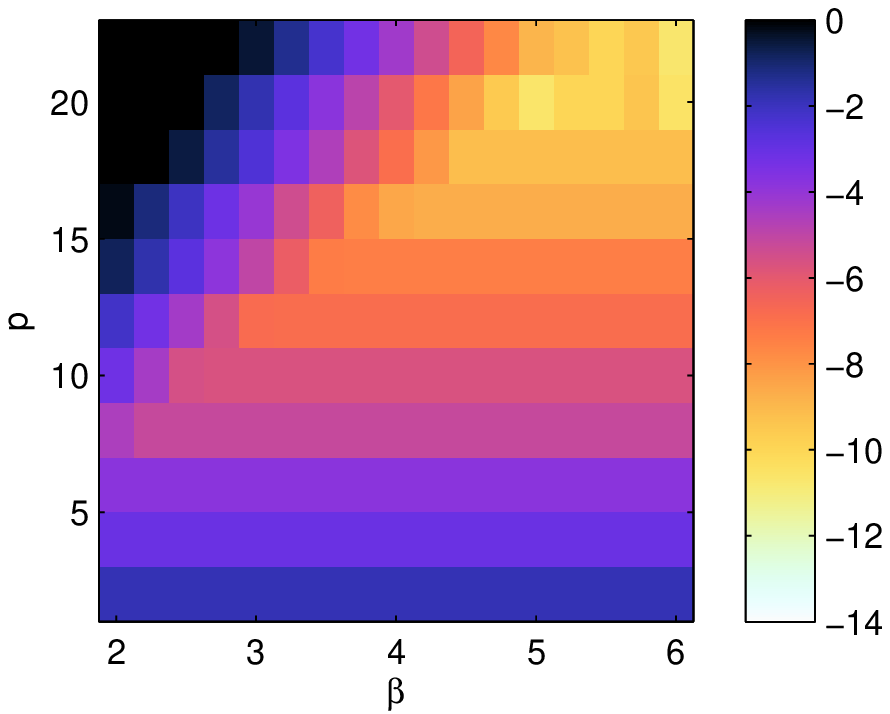}
\emp
\bmp{2in}(c)
{\small $\log_{10} \|\hat u-u\|_\infty/\|u\|_\infty$, \, $N{=}340$\\
\mbox{}\qquad $u$ pt src, \;\; trig poly interp}\\
\ig{width=2in}{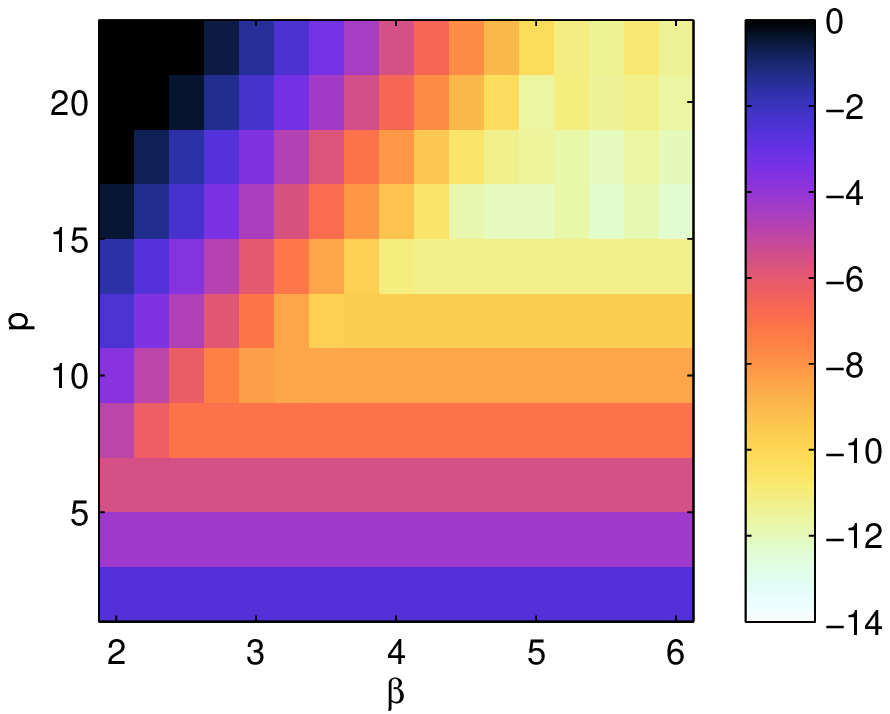}
\emp
\ca{Convergence of the maximum relative $L^\infty(\Omega)$ error with
respect to both parameters $p$ and $\beta$, at fixed $N$ and $\delta$,
for the surrogate evaluation scheme of the double-layer potential solving a
Laplace BVP in the domain
of Fig.~\ref{f:rfnk0}.
(a) Dirichlet data $u(x,y) = e^y \cos x$, with $N=130$,
trigonometric interpolant.
(b) $u(x,y) = \log \|(x,y)-(1,0.5)\|$, with $N=180$,
\Ny\ interpolant.
(c) $u(x,y) = \log \|(x,y)-(1,0.5)\|$, with $N=340$,
trigonometric interpolant.
}{f:rfnk0conv}
\efi

\subsection{Numerical performance and the effect of a nearby singularity}
\label{s:sing}

We now study in more detail the
performance of the surrogate method for the Laplace
double-layer potential, in the context of solving a Dirichlet
BVP using the \Ny\ method.

We first return to the case presented in Fig.~\ref{f:rfnk0}:
boundary data $u(x,y) = e^y \cos x$, an entire function.
We fix $N=130$ nodes (around the value
at which the \Ny\ method has completely converged),
and use a trigonometric polynomial interpolant to
get $\tau$ at the fine nodes.
With the $N_B$ boxes fixed, we
consider convergence with respect to $p$ and $\beta$.
Fig.~\ref{f:rfnk0conv}(a) shows the resulting $L^\infty$ errors (estimated on a
spatial grid of spacing 0.02) relative to $\|u\|_\infty$.
The two terms in Theorem~\ref{t:eps} are clearly visible:
the errors are always large at small $p$ 
(due to the first term in \eqref{eps}),
but even at large $p$ the errors are large when $\beta$ is too small
(due to the second term).
The $\beta$ needed for convergence grows roughly linearly in $p$,
as one would expect if the value of the log in \eqref{2nd} is
treated as roughly constant.
In both directions convergence appears exponential, exceeding 13 digits
once $p=10$ and $\beta\ge 4$.

\begin{rmk}
For $p\ge 16$ the best achievable error
(once $\beta$ has converged) worsens slightly
as $p$ grows.
We believe that this is due to catastrophic cancellation in the oscillatory
integrand \eqref{cmd} for large $m$,
combined with the usual double-precision round-off error.
This seems to be a fundamental limit of the local expansion
surrogate method implemented in floating-point arithmetic.
However, the loss is quite mild: even at $p=22$ it is only 2.5 digits.
\label{r:cancel}
\end{rmk}

We now change to boundary data $u(x,y) = \log \|(x,y)-(1,0.5)\|$,
which is still a Laplace solution in $\Omega$ and is still
real analytic on $\pO$, but whose analytic continuation
outside $\Omega$ has a {\em singularity} at the exterior point $(1,0.5)$,
a distance of only 0.24 from $\pO$.
Its preimage has distance from the real axis
$\al_\ast := -\im Z^{-1}(1+0.5i) \approx 0.176$.
The \Ny\ method fully converged by $N=180$, as assessed by
the error at a distant interior point and by the density values
at the nodes $\tau_j$: see the first two curves in Fig.~\ref{f:rfnk0interp}(a).
In fact the convergence rate matches $e^{-\al_\ast N}$,
as is to be expected, since
the \Ny\ convergence rate is known to be the same as that of the
underlying quadrature scheme (see discussion after \cite[Cor.~12.9]{LIE}),
which here is controlled by the singularity via the Davis theorem.

However, turning to surrogate evaluation, Fig.~\ref{f:rfnk0conv}(b)
shows that the error due to Taylor
truncation converges quite slowly with $p$, as is inevitable for a nearby
singularity, pushing the optimal $p$ up to around 22. In the best case only
10 digits are achieved.
To evaluate $\tau$ at the fine nodes, the \Ny\ interpolant \cite[Sec.~12.2]{LIE}
was used here rather than the trigonometric polynomial
interpolant, for the following reason.

\begin{rmk}[interpolants]
In the regime where a nearby singularity in the 
right-hand side data
controls the convergence rate (rather than the kernel function),
the \Ny\ interpolant converges {\em twice as fast}
as the trigonometric polynomial interpolant,
as shown by Fig.~\ref{f:rfnk0interp}(a).
This reflects the 
fact that periodic trapezoid quadrature
is ``twice as good'' as trigonometric interpolation \cite[p.201]{LIE},
because the former is exact for Fourier components with index
magnitudes up to $N-1$, but the latter is exact only up to $N/2$.
Informally speaking, the trapezoid rule (and hence the
\Ny\ interpolant) ``beats the Nyquist sampling theorem by a factor of two!''
\label{r:twice}
\end{rmk}

Thus the \Ny\ interpolant is preferred in this context
when it is desired that $N$ be its smallest converged value.
Fig.~\ref{f:rfnk0interp}(b) and (c) show the loss in accuracy
(and spurious evanescent waves which appear near $\pO$)
that result from attempting to use the inferior trigonometric interpolant.
Note that this
loss would occur for {\em any} accurate close-evaluation scheme,
since it is a loss of accuracy in the function $\tau$ itself.
Nor is the \Ny\ interpolant perfect: it requires the application of
a $\Nf$-by-$N$
dense matrix, and it appears to cause up to 1 digit more roundoff error
than in the trigonometic case.

However, by nearly doubling $N$ to 340 (which is somewhat wasteful),
the Fourier coefficents
of $\ttau$ decay to machine precision by index $N/2$
(see Fig.~\ref{f:rfnk0interp}(a))
making the trigonometric interpolant as accurate as the \Ny\ one,
whilst box radii decrease so that the $p$-convergence is faster.
We show convergence for this trigonometric case in Fig.~\ref{f:rfnk0conv}(c):
12 digits accuracy result at $p=16$ and $\beta=5$.

\bfi 
\hspace{-.4in}
\bmp{3.1in}
\ig{width=3.1in}{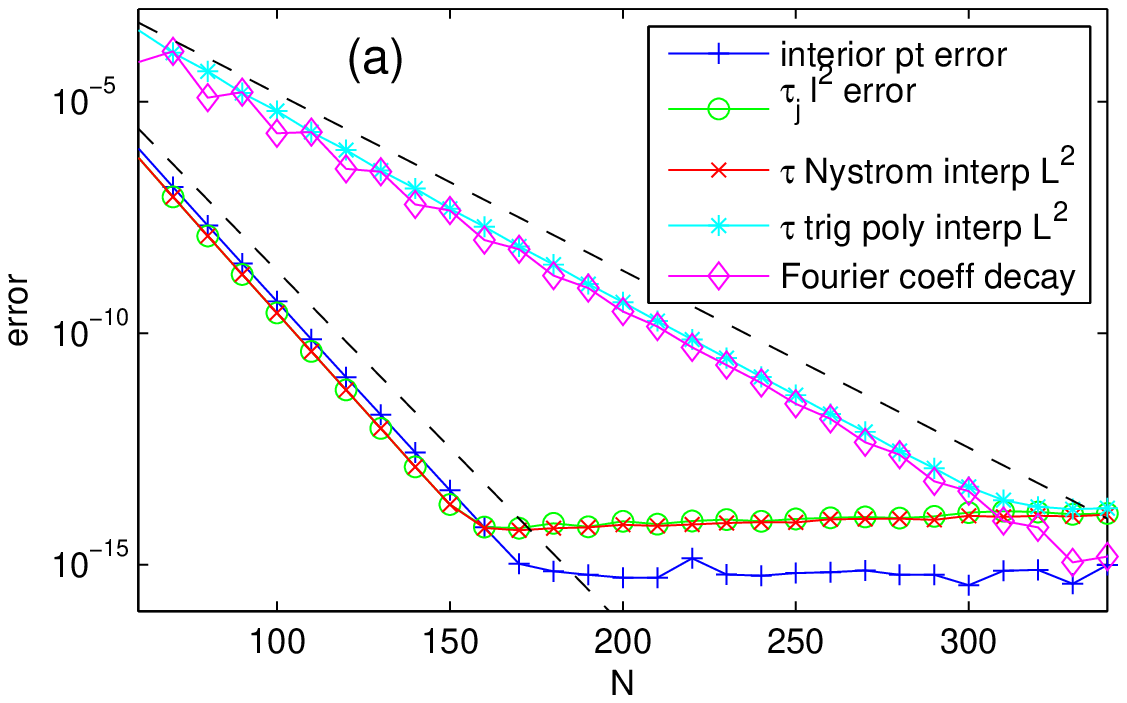}
\emp
\bmp{1.5in}(b)
{\small $\log_{10} \|\hat u-u\|_\infty/\|u\|_\infty$\\
\mbox{}\qquad \Ny\ interp}\\
\ig{width=1.5in}{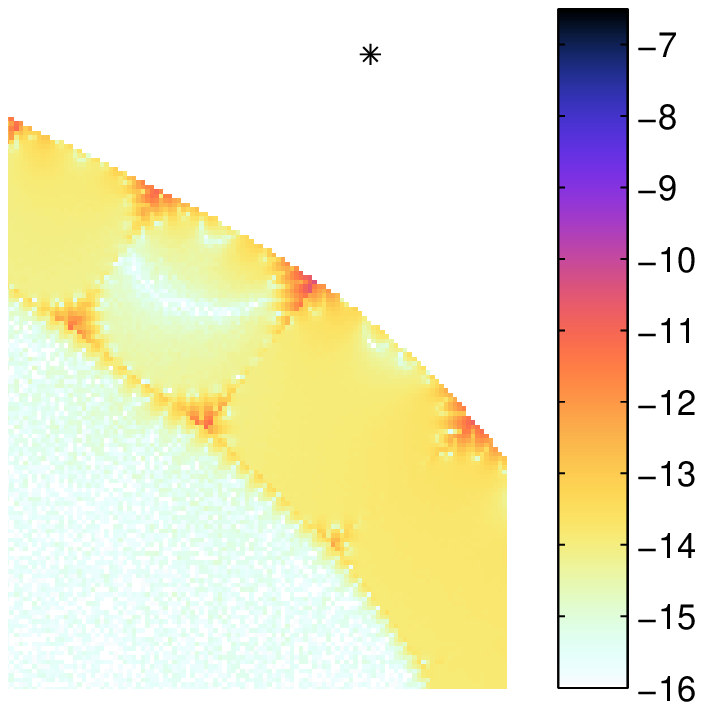}
\emp
\bmp{1.5in}(c)
{\small $\log_{10} \|\hat u-u\|_\infty/\|u\|_\infty$\\
\mbox{}\qquad trig interp}\\
\ig{width=1.5in}{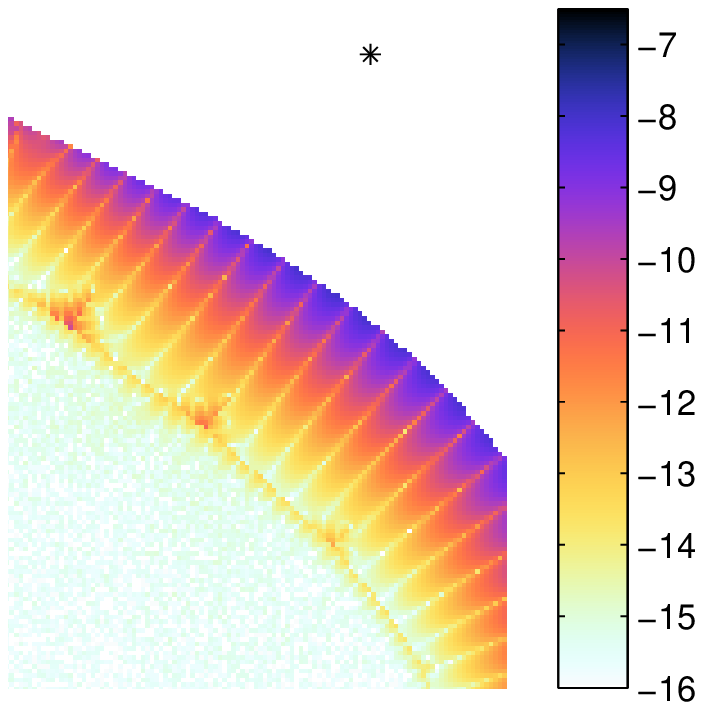}
\emp
\ca{
Comparing two interpolants in the \Ny\ solution of $\tau$,
for  solution of the Dirichlet BVP for Laplace's equation,
with data $u(x,y) = \log \|(x,y)-(1,0.5)\|$,
in the domain of Fig.~\ref{f:rfnk0}.
(a) shows convergence of various errors:
the solution at a distant interior point, the $l_2$ error at the nodes,
the $L^2([0,2\pi))$ errors for the \Ny\ and trigonometric interpolants,
and the size of the $N/2$ Fourier coefficient of $\tau$.
Dotted lines show exponential decay at the rates $e^{-\al_\ast N}$
and $e^{-\al_\ast N/2}$.
Fixing $N=180$, a zoomed plot of the relative error in the surrogate
scheme is shown using
(b) \Ny\ interpolant, and (c) trigonometric interpolant,
near the singularity (shown by a $\ast$).
}{f:rfnk0interp}
\efi


\subsection{The Laplace single-layer case and a Neumann problem}
\label{s:slpthm}

Recall that the single-layer potential is the real part of $v$ given
by \eqref{vs}.
By writing $\log 1/(y-z) = \log 1/(y-z_0) - \log(1-\frac{z-z_0}{y-y_0})$,
and using the Taylor series for the second logarithm, we find that
the single-layer version of \eqref{cmd} is
\be
c_0 \;=\; \frac{1}{2\pi}\int_\pO \biggl(\log \frac{1}{y-z_0}\biggr)
\sigma(y) |dy|,
\qquad
c_m\;=\;\frac{1}{2\pi m} \int_\pO \frac{\sigma(y)}{(y-z_0)^m} |dy|,
\quad m=1,2,\ldots
\label{cms}
\ee
Here for $c_0$ the branch cuts of Remark~\ref{r:branch} apply.
The convergence of the surrogate scheme
is then as least as good as for the double-layer
case.
\begin{thm} 
The version of Theorem~\ref{t:eps} corresponding to the single-layer potential
holds.
That is, with the same conditions, the uniform error bound \eqref{eps} holds,
but with $\tau$ changed to $\sigma$, the potential $v$ given by
\eqref{vs}, and $\hat c_m$ given by the $M$-node periodic trapezoid
rule applied to the parametrized version of \eqref{cms}.
\label{t:epss}
\end{thm} 
\begin{proof}
Using \eqref{cms} and the Davis theorem,
the coefficient errors analogous to
\eqref{cmerr} are
$$
|\hat{c}_0-c_0| \le C \bigl|\log d(\Gamma_\al,z_0)\bigr|
\frac{1}{e^{\al \Nf} - 1}
~, \qquad
|\hat{c}_m-c_m| \le 
\frac{C}{m\,d(\Gamma_\al,z_0)^{m}} \frac{1}{e^{\al \Nf} - 1}
~,\quad m=1,2,\ldots
$$
But for each $m=0,1,\ldots$, this error is smaller than that in \eqref{cmerr},
after possibly a change in the constant $C$.
The rest of the proof follows through.
\end{proof}
\begin{rmk}
Note that the change in the constant $C$ referred to 
is generally in the favorable direction:
since $d(\Gamma_\al,z_0) < R$, the constant
may be multiplied by a factor given by the larger of $R$ and $R|\log R|$,
which are usually small.
Also note that, since factors of $1/m$ arise in the single-layer case,
but were not taken advantage of, the $p$-convergence could probably be
improved slightly.
\label{r:slpbetter}
\end{rmk}

As an application,
we now report numerical results for the interior Laplace--Neumann BVP
\bea
\Delta u & = & 0 \qquad \mbox{ in }\Omega
\label{lapneu1}
\\
\partial u / \partial n & = & f \qquad \mbox{ on }\pO
\label{lapneu2}
\eea
which has a solution only if $f$ has zero mean on $\pO$, and in that
case the solution is unique only up to an additive constant.
Following \cite[Sec.~7.2]{atkinson} we use the integral equation
\be
(D^\ast + K + \half I)\sigma = f~,
\label{bieneu}
\ee
where $D^\ast$ has kernel $k(x,y) =
\partial \Phi(x,y)/\partial n(x)$, and
$K$ is the boundary operator which returns the value of its operand
at some (fixed but arbitrary) point on $\pO$. 
The solution is then recovered up to an unknown constant by \eqref{slp};
we compare against the exact interior solution after subtracting the
value of the constant measured at a single point.

As with the Dirichlet case,
we test with boundary data coming from the entire function $u(x,y)=e^y \cos x$,
or from the function
with a nearby singularity $u(x,y) =  \log \|(x,y)-(1,0.5)\|$.
In the entire case, we find, as for the Dirichlet case, that
the convergence of the \Ny\ method saturates at around $N=130$.
Surrogate expansion (using the \Ny\ interpolant)
then gives a maximum relative error of $6\times 10^{-15}$ at $p=10$ and
$\beta=4$.%
\footnote{Curiously, the trigonometric interpolant in this
case does not become fully accurate until $N=340$, which is {\em much} more
than the value $N=130$ for the Dirichlet case.
We cannot explain this---neither result is
as predicted by Remark~\ref{r:twice}---and
it tells us that there is still more to understand about the \Ny\ method
convergence rates in analytic BVP settings.}
For the singularity case, convergence of the \Ny\ method
is complete at around a value $N=200$
and the maximum relative surrogate evaluation error is then
found to be $6\times 10^{-13}$ at $p=24$ and $\beta=5.5$.
Notice that, for both data types,
these single-layer errors are improved by at least one digit over the
double-layer errors reported in Sec.~\ref{s:sing};
this may be explained by Remark~\ref{r:slpbetter}.
Indeed, the convergence plots analogous to
Fig.~\ref{f:rfnk0conv}(a) and (b) are very similar but show a gain of
around 1 extra digit of accuracy.

\section{The Helmholtz equation and an $O(N)$ close evaluation scheme}
\label{s:helm}

We now move to a PDE for which we no longer have
theorems, but which has important applications.

\bfi 
\qquad
\bmp{2.8in}(a)
{\small $\log_{10} |u^{(N)}-u|/\|u\|_\infty$ \;\; ext BVP, $u$ pt src}\\
\ig{width=2.8in}{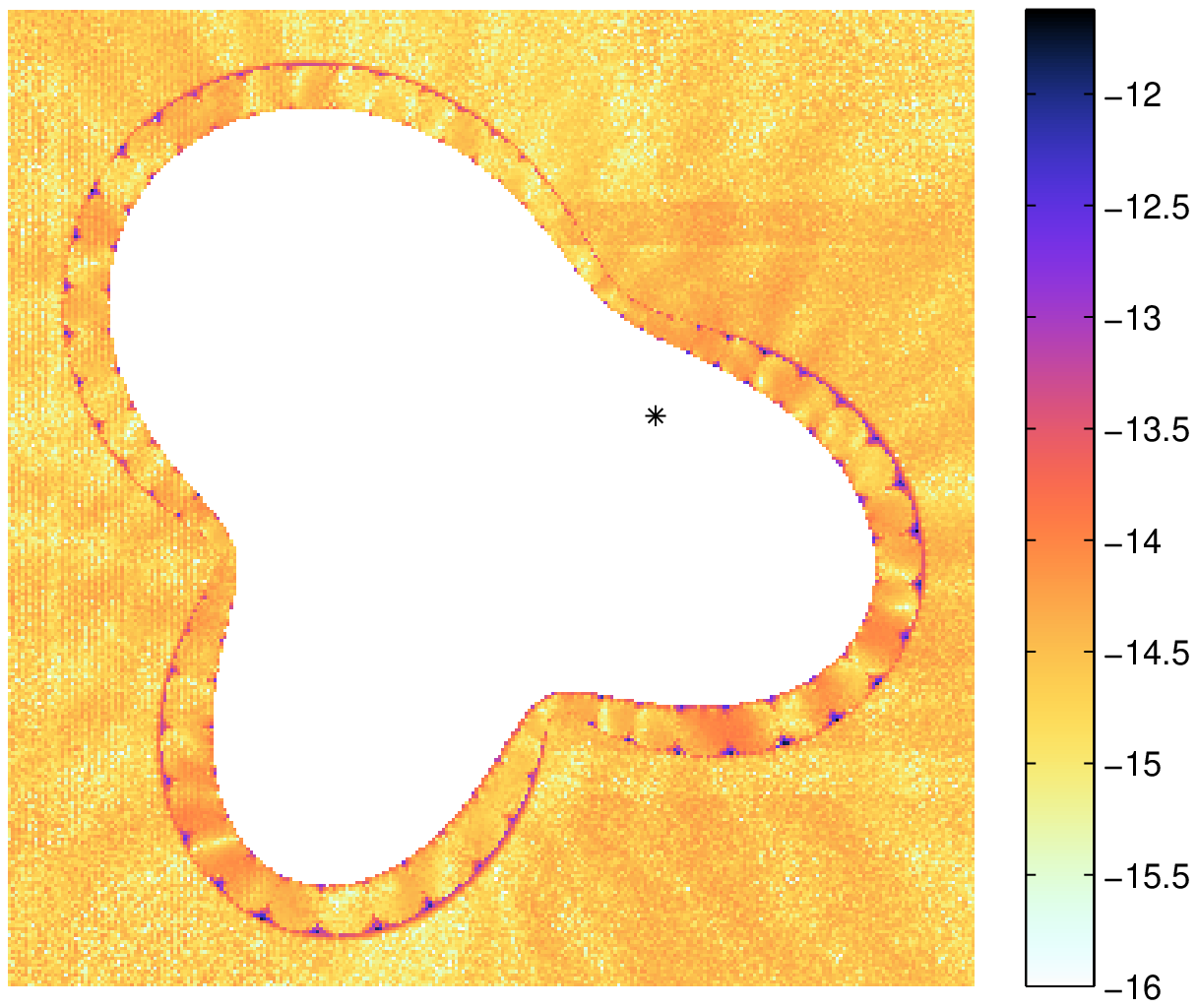}
\emp
\qquad
\bmp{2.2in}(b)
{\small $\log_{10} \|\hat u-u\|_\infty/\|u\|_\infty$, \;\; $N=340$\\
\mbox{}\qquad $u$ pt src, \;\; trig poly interp}\\
\ig{width=2in}{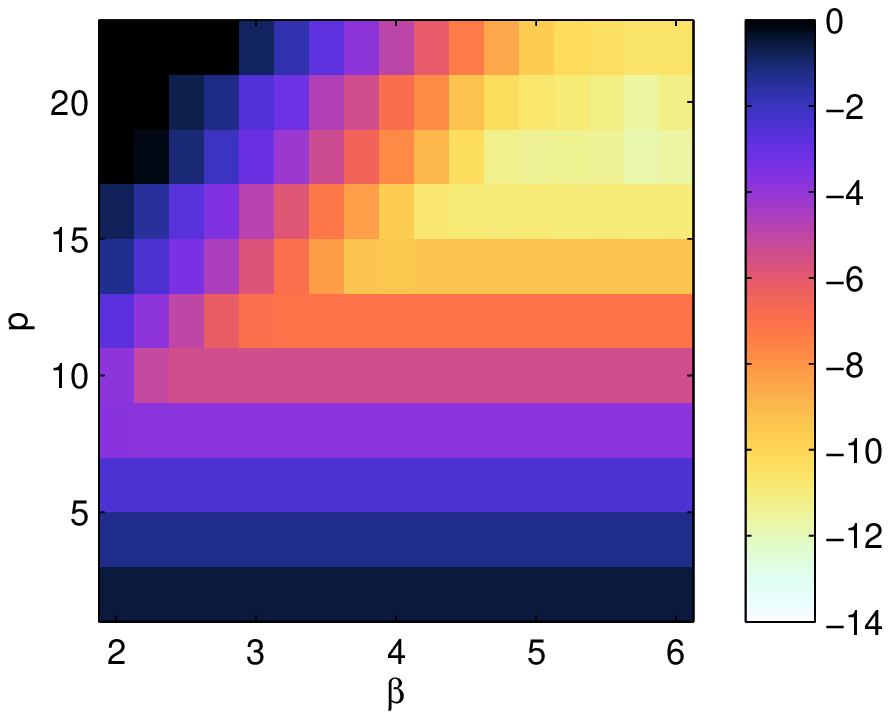}
\vspace{.3in}
\mbox{}
\emp
\ca{Relative error for the surrogate
scheme for the exterior Helmholtz Dirichlet problem at $\om=30$
(the diameter is 12 wavelengths),
with known solution $\Phi(\cdot,x_0)$ with $x_0$ as shown by the $\ast$ symbol.
(a) error plot for $N=340$, $p=18$ and $\beta=6$
(many boxes are apparent).
(b) $L^\infty(\Omega)$ error convergence with respect to parameters
$p$ and $\beta$, at fixed $N=340$.
}{f:rfnk30}
\efi

\subsection{Implementation and convergence test for the Helmholtz equation}
\label{s:implhelm}

The above close evaluation scheme for $\om=0$ with real-valued potentials
is very easily adapted to the Helmholtz
equation ($\om>0$) with complex potentials, by replacing a couple of
formulae.
The $p$-term Taylor expansion \eqref{vqbx} for $v$
(whose real part was taken to get $u$),
is replaced by
a local (Fourier-Bessel) expansion with $2p{-}1$ terms,
$$
\hat u(z) = \sum_{|m|<p} c_m e^{im\theta} J_m(\om r)
~,
\qquad
\mbox{ where } \; z-z_0 = r e^{i\theta}, \qquad z \mbox{ in box } B
~,
$$
i.e.\ $(r,\theta)$ is the polar coordinate system with origin $z_0$.
For the single-layer potential,
recalling \eqref{Phi} and Graf's addition formula
\cite[(10.3.7)]{dlmf}, the formula \eqref{cms} is replaced by
\be
c_m \;=\; \frac{i}{4} \int_\pO e^{-im\theta_y} H^{(1)}_m (\om r_y) \sigma(y) ds_y
~, \qquad |m|<p~,
\label{hcms}
\ee
where $(r_y,\theta_y)$ are the polar coordinates of the point $y$
relative to the origin $z_0$.
Using the addition formula,
the reflection formulae, after some simplification,
the Cauchy formula \eqref{cmd} is replaced by
\be
c_m \;=\; \frac{i\om}{8} \int_\pO \left[
e^{-i(m-1)\theta_y - i\nu_y} H^{(1)}_{m-1}(\om r_y) -
e^{-i(m+1)\theta_y + i\nu_y} H^{(1)}_{m+1}(\om r_y)  
\right] \tau(y) ds_y
~, \qquad |m|<p~,
\label{hcmd}
\ee
where $\nu_y$ is the angle of the outward normal at $y\in\pO$.
The use of the above three formulae is all the change needed
to make the Helmholtz version of the scheme.

We have already seen in Fig.~\ref{f:err}(d) that
the Laplace equation predictions for the native evaluation error
(theorems \ref{t:convd} and \ref{t:convs})
also hold well for the low-frequency Helmholtz equation.
This is to be expected,
since for $\om>0$ the fundamental solution \eqref{Phi}
remains analytic away from the origin,
so the Davis theorem 
applies to the parametrized
\eqref{dlp} and \eqref{slp},
allowing the Laplace convergence rate to be approached.
(We leave the Helmholtz equivalents of
the tight theorems \ref{t:convd} and \ref{t:convs} for future work.)
Therefore we apply the same criterion for the bad annular neighborhood
$\Omega_\tbox{bad}$
as in section~\ref{s:surr}.

We test the scheme in the context of BVP applications, namely the
Dirichlet problem in an exterior domain $\Omega$,
with the usual Sommerfeld radiation condition \cite[(3.62)]{coltonkress},
for which the integral equation \eqref{bie}
is replaced by the so-called combined-field formulation
\cite[p.~48]{coltonkress}
\be
(D - i\om S + \half I) \tau = f
\label{biecfie}
\ee
which is well conditioned for all $\om>0$.
Here the potential is represented as
\be
u(x) = \int_\pO \left[ \frac{\partial \Phi(x,y)}{\partial n(y)}
- i\om \Phi(x,y) \right]
\tau(y) ds_y ~,
\qquad x\in\Omega~,
\label{cfie}
\ee
and we choose the right-hand side data as coming from an interior
point-source shown by the $\ast$ in Fig.~\ref{f:rfnk30}(a).
To achieve spectral accuracy in the \Ny\ method, special
quadratures for the weak logarithmic singularity are needed;
we use the scheme of Kress \cite{kress91} (also see \cite{hao}).

We fix $\om=30$ (thus $\pO$ is 12 wavelengths across),
and find that the \Ny\ method converges at around $N=340$.
A maximum relative surrogate
evaluation error around $3\times 10^{-12}$ is then achieved
at $p=18$, $\beta=6$, as shown in Fig.~\ref{f:rfnk30}(a).
This is dominated by errors in the box corners furthest from
$\pO$, especially those with distortion due to a convex part of $\pO$.
Since this holds largely independently of the singularity location $x_0$,
we believe it is instead controlled by the wavenumber $\om$.
The $L^\infty(\Omega)$ convergence with $p$ and $\beta$ is
very similar to the Laplace case, as shown in
Fig.~\ref{f:rfnk30}(b). $p=18$ is close to optimal, since for
$p>20$ the error
worsens---the explanation is believed to be as in Remark~\ref{r:cancel},
but demands further study.
However, simply by making the boxes slightly narrower and more numerous
by setting $N_B = \lceil N/4 \rceil$, with 25\% more effort
we cut the maximum relative error to $3\times 10^{-13}$.




\begin{rmk}[interpolants revisited]
A natural question is: does Remark~\ref{r:twice} hold for the Helmholtz
equation? The answer is no, at least for the Kress scheme
(which is one of the best known \cite{hao}).
Although a \Ny\ interpolant does exist for the weakly-singular
kernels \cite[(3.3)]{kress91}, we believe it has little advantage
over the trigonometric interpolant.
This is because the product-quadrature scheme of Kress
is only accurate for Fourier components of indices up to $N/2$.
In data with a nearby singularity with preimage at a distance $\al_\ast$
from the real axis,
the convergence rate is thus only $e^{-\al_s N/2}$, or
half that of the Laplace case.
\label{r:helmtwice}
\end{rmk}

\subsection{Evaluation via a correction to the fast multipole method}
\label{s:ON}

At higher frequencies $\om$ and/or complex geometries $\pO$,
the value of $N$ needed for convergence of the \Ny\ method
is pushed higher.
For instance, with fixed geometry the high frequency asymptotics is
empirically $N = O(\om)$, i.e.\ a constant number of nodes per wavelength
\cite{mfs,helsingtut}.
The surrogate scheme as presented requires $O(p N^2)$ effort, by
evaluating $p$ coefficients at $N_B = O(N)$ centers using
$M = O(N)$ fine node kernel evaluations for each.
Clearly, a better scaling with $N$ would be preferred.
We now show that at fixed frequency,
linear scaling is easy to achieve by locally correcting
the FMM.

Fixing a box $B$ with center $z_0$,
and taking e.g.\ the Helmholtz single-layer potential \eqref{slp},
we split the integral into ``near'' and
``far'' parts using a cut-off radius $G$.
We apply the surrogate local expansion only for the near part to get
\be
\hat u(x) \;=\;
\sum_{|m|<p} c_m^\tbox{(near)}
e^{im\theta} J_m(\om r)
\;+\;
\int_{y\in\pO,|y-z_0|> G} \!\! \Phi(x,y) \sigma(y) ds_y
~, \qquad x \mbox{ in box } B~,
\label{slpsplit}
\ee
where, adapting \eqref{hcms},
the coefficients of the potential due to the near part of the
integral only are,
\be
c_m^\tbox{(near)} = \frac{i}{4} \int_{y\in\pO,|y-z_0|\le G} \!\!\!
e^{-im\theta_y} H^{(1)}_m (\om r_y) \sigma(y) ds_y
\approx
\frac{i\pi}{2M} \sum_{j\in J_\tbox{near}}
e^{-im\theta_{Z(s_j)}} H^{(1)}_m (\om r_{Z(s_j)}) \tsig(s_j) Z'(s_j)
~,
\label{hcmsnear}
\ee
and where $J\tbox{near}:=\{j: 
\,|Z(s_j)-z_0|\le G\}$
is the index set of the ``near'' subset of the fine nodes $\{s_j\}_{j=1}^M$.
The cut-off $G$ must be large enough that the fictitious singularities
induced at the ends of the near interval are distant enough
not to slow down the convergence of the local expansion,
thus we choose $G$ several times the box radius $R$,
and so evaluating these sums takes $O(p)$ effort.
Applying the fine quadrature to the far part of the integral in \eqref{slpsplit}
gives
\be
\int_{y\in\pO,|y-z_0|> G} \!\! \Phi(x,y) \sigma(y) ds_y \;\approx\;
\frac{2\pi}{M} \sum_{j=1}^M \Phi(x,Z(s_j)) \tsig(s_j) Z'(s_j)
-
\frac{2\pi}{M} \sum_{j\in J_\tbox{near}} \Phi(x,Z(s_j)) \tsig(s_j) Z'(s_j)
~.
\label{fmmcorr}
\ee
The first sum can be evaluated {\em for all target points}
$x$ {\em in all boxes}
in a single FMM call. Making the reasonable assumption
that the user demands $O(1)$ targets per box, this FMM call
requires $O(N)$ effort.
Finally, the second sum is a local correction that takes $O(1)$ effort.
The total effort to evaluate $\hat u(x)$ for all $x$ in all
$N_B$ boxes is thus $O(pN)$. A similar scheme applies for the double-layer
potential.

Since the first (FMM) term in \eqref{fmmcorr} accurately approximates
$u$ everywhere except in a narrower neighborhood $A_\al$
with $\al=10\pi/\beta N$,
we in fact may, and will, shrink the boxes in the normal direction, as long as they cover $A_\al$,
whilst using this first term in the remaining part of $\Omega_\tbox{bad}$.
This has the advantage of avoiding larger errors that tend to
occur in the distant corners of boxes.

\begin{rmk}[no end corrections]
At this point the
reader might very well suspect that, since both of the above integrals
are on 
non-periodic intervals, the trapezoid rule would give at best
low-order $O(1/M)$ convergence unless higher-order
end correction rules were used.
In fact, spectral accuracy equal to that of the trapezoid rule on
the original periodic integrand is observed.
The reason is slightly subtle: 
the endpoint errors
in the two integrals cancel, since their sum ultimately represents
to high order
the result of the trapezoid rule applied to the periodic analytic integrand in
\eqref{slp}.
Thus our simple splitting
achieves spectral accuracy, needing
neither a partition of unity nor end-point quadrature corrections.
\end{rmk}

\bfi[ht!] 
\hspace{-.2in}\bmp{3.5in}(a)
{\small $u_\tbox{tot}$}\\
\ig{width=3.5in}{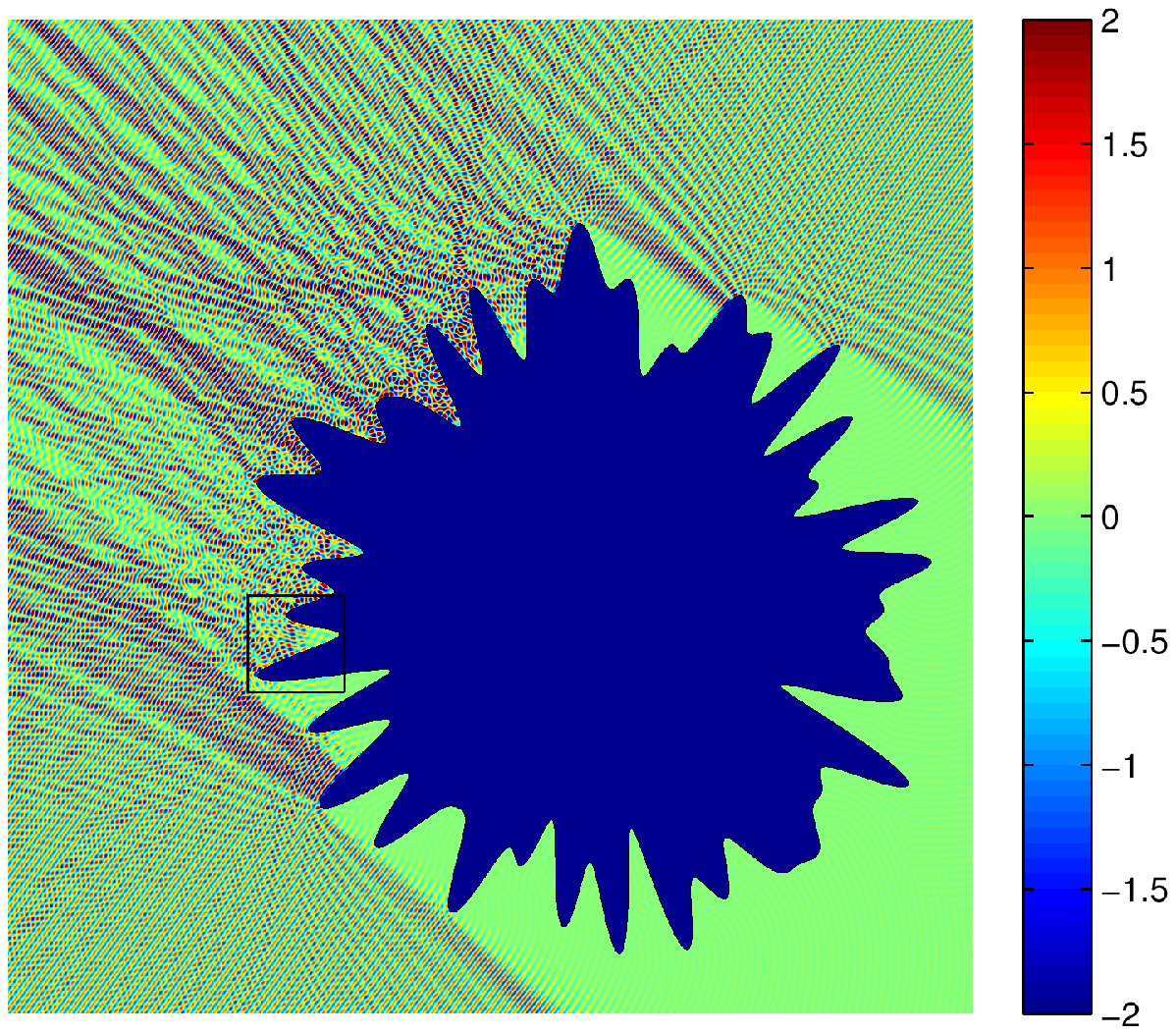}
\emp
\quad
\bmp{2.35in}(b)
{\small $u_\tbox{tot}$ zoomed}
\vspace{.08in}

\ig{width=2.35in}{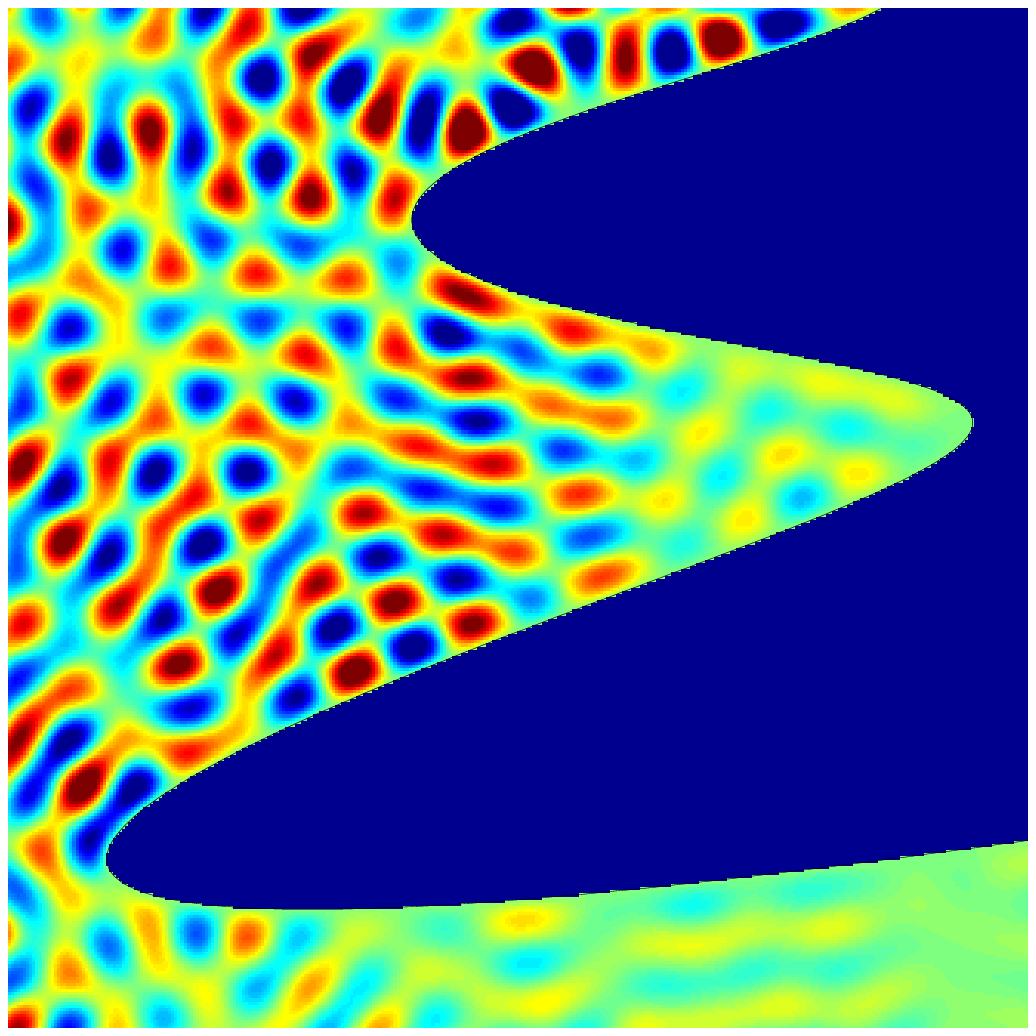}
\vspace{.4in}

\mbox{}
\emp
\vspace{.1in}

\hspace{-.2in}\bmp{3in}(c)
{\small $\log_{10} |u^{(N)}-u|/\|u_\tbox{tot}\|_\infty$}\\
\ig{width=3in}{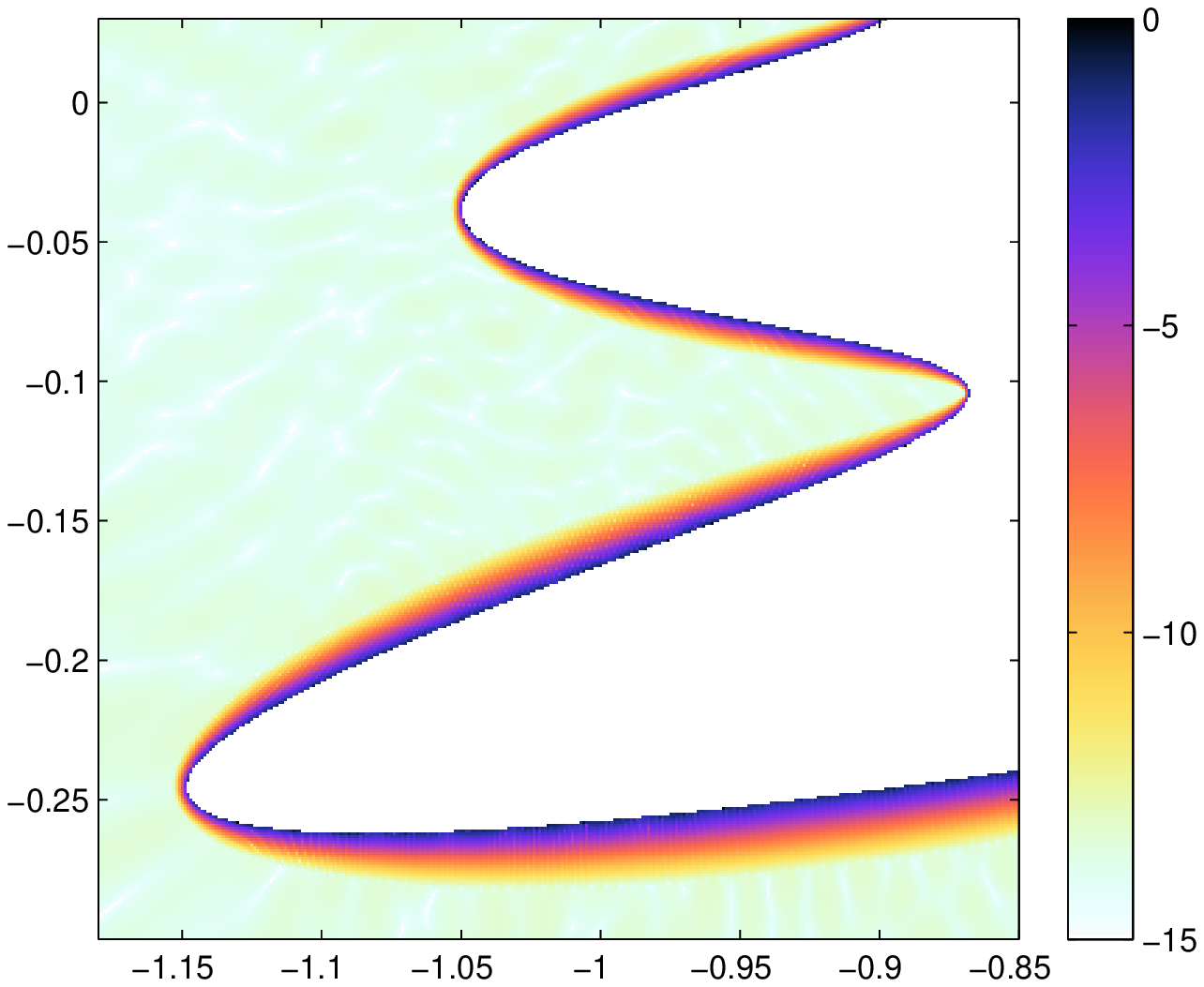}
\emp
\quad
\bmp{3in}(d)
{\small $\log_{10} |\hat u-u|/\|u_\tbox{tot}\|_\infty$}\\
\ig{width=3in}{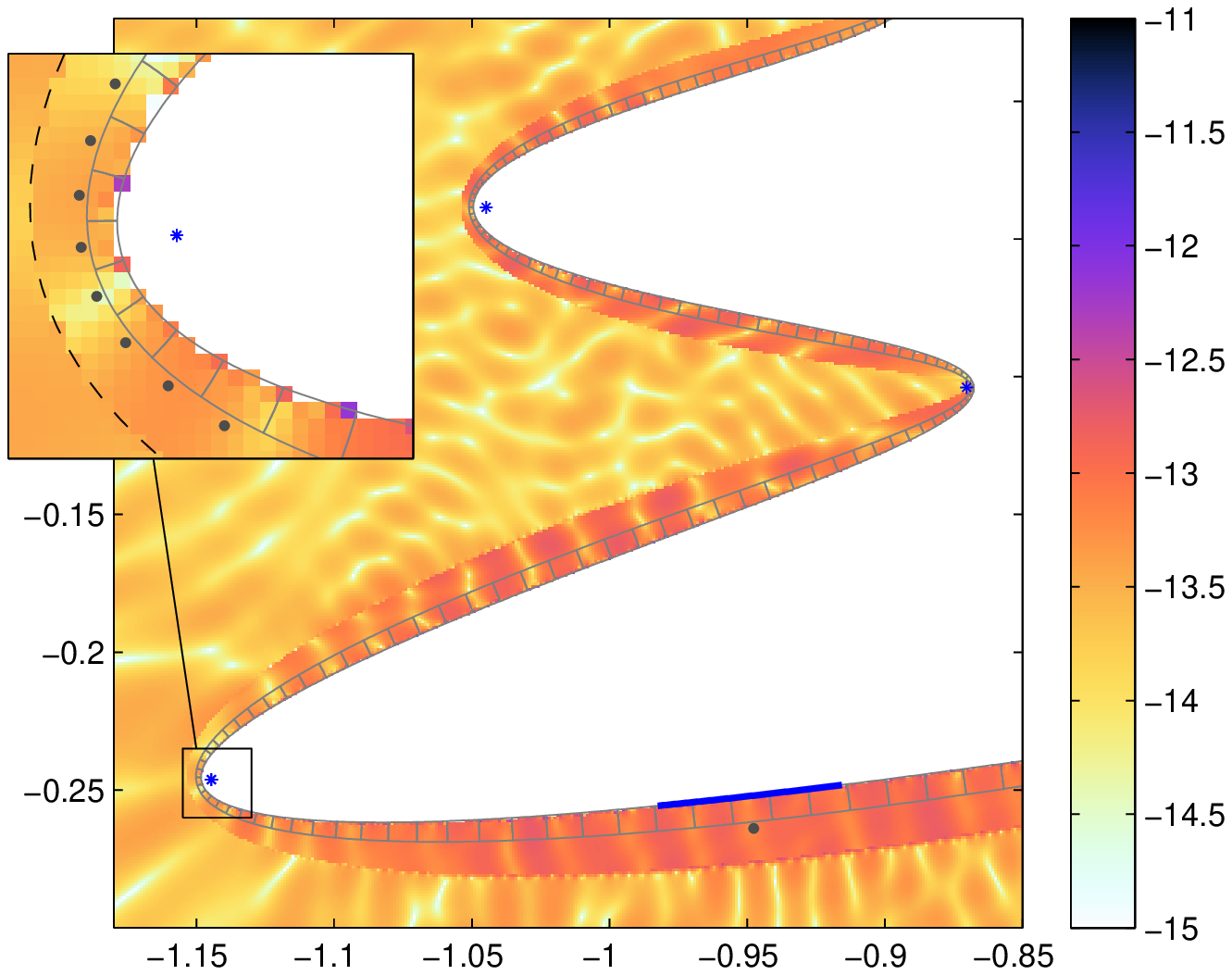}
\emp
\vspace{.1in}
\ca{High frequency sound-soft
scattering example, $\om=250$ (diameter is 100 wavelengths),
needing $N=9000$ nodes.
(a) Total field $u_\tbox{tot} = u_\tbox{inc} + u$.
(b) Zoom of total field for black box shown in (a).
(c) Relative error (relative to $\|u_\tbox{tot}\|_\infty=6.2$)
in native evaluation $u^{(N)}$ in same region as (b).
(d) Relative error in surrogate scheme $\hat u$ in same region as (b),
with $p=26$, $\beta = 6$.
Boxes are shown in grey, and a single center $z_0$ (grey dot)
with its corresponding ``near'' set of boundary points (thick blue line),
and domain Schwarz singularities ($\ast$).
Note the change in color scale between (c) and (d).
The inset in (d) is a zoom of the highly convex region,
and also shows all centers (grey dots) and
$\Gamma_{\al_\tbox{bad}}$, the boundary of $\Omega_\tbox{bad}$ (dotted line).
}{f:high}
\efi

\subsection{High frequency scattering example}
\label{s:high}

We now detail the application of the
$O(N)$ method just described to a high frequency scattering problem
with Dirichlet boundary condition; this corresponds to
an acoustically sound-soft obstacle.
We choose a complicated (but analytic) boundary $\pO$ given by the polar
Fourier series
$f(\theta) = 1 + \sum_{n=1}^{40} a_n \cos(n\theta) + b_n \sin(n\theta)$,
with $a_n$, $b_n$ uniform random in $[-0.04,0.04]$, and parametrized
by $\theta$; see Fig.~\ref{f:high}(a).
We choose $\om=250$ such that the obstacle is 100 wavelengths
across.
For the scattering problem with plane wave
$u_\tbox{inc}(x) = e^{i\omega \hat d \cdot x}$ incident at angle
$\hat d = (\cos -\pi/5,\sin -\pi/5)$,
the total potential (physical field)
is $u_\tbox{tot} = u_\tbox{inc} + u$ where the (radiative)
scattered potential $u$ solves the exterior Dirichlet problem with
boundary data $f = -u_\tbox{inc}|_\pO$. For this BVP, as before,
we solve the integral equation \eqref{biecfie} then evaluate $u$
via \eqref{cfie}.

We find that $N=9000$ is needed to get 13-digit convergence of the \Ny\ method
for $u$,
as assessed at a variety of
exterior points lying outside the bad annular neighborhood $\Omega_\tbox{bad}$.
This corresponds to an average of 13 nodes per wavelength, although it
is as little as 4.6 nodes per wavelength where $|Z'(s)|$ is largest.
It takes 73 s to fill the dense $N\times N$ matrix,
and a further 14 s to solve the system via GMRES
using dense matrix-vector products
(needing 95 iterations to exceed a relative residual of $10^{-12}$).

\begin{rmk} Here all timings are reported for a System76 laptop
with 2.6 GHz Intel i7-3720QM CPU, 16GB of RAM, running
\matlab\ 2012b \cite{matlab},
\mpspack\ version 1.32 \cite{mpspack},
and {\tt FMMLIB2D} version 1.2 \cite{hfmm2d}.
We use \matlab's native Bessel functions.
For Hankel functions in \eqref{hcmsnear} and the second sum
in \eqref{fmmcorr} (and their double-layer analogs)
we use a MEX interface to
{\tt hank103.f} \cite{hfmm2d} for $m=0,1$, and upwards recurrence
\cite[(10.6.1)]{dlmf}.
Most operations use only a single core; the only ones which
exploit all four cores are {\tt FMMLIB2D} and the matrix-vector
products in GMRES.
\end{rmk}

We fix a set of around $8\times 10^6$ evaluation points, namely those
on a $3300\times 3400$
grid of spacing $10^{-3}$ which lie in the exterior of $\pO$;
see Fig.~\ref{f:high}.
Evaluation of $u$ at these targets using the $N$ native nodes and the FMM
takes 24 s, and gives the relative errors in Fig.~\ref{f:high}(c):
large errors are apparent near $\pO$.
We now apply the surrogate close evaluation to the
points lying in $\Omega_\tbox{bad}$, which number
around $2.2\times 10^5$.
We set $N_B=\lceil N/3\rceil = 3000$,
rather more than the value $\lceil N/5\rceil$ recommended before: this
helps reduce errors by shrinking the box radii $R$.
We find convergence at around $p=26$, $\beta=6$
(this high $p$ is needed because boxes are up to 0.7 wavelengths in size).
A cut-off radius $G$ that does not induce additional error was found to
be 1.5 times the maximum width of
the bad annular neighborhood in the surrounding 7 boxes;
this gives between 84 and 326 ``near'' fine points,
but with a mean of only 107 (note that this is less
than the 288 that would be required for three 16-node 
panels and the same $\beta$; see remark \ref{r:panels}).
With the above parameters we achieve a maximum
error in $u$ of
$6 \times 10^{-12}$, in 28 s computation time.
Much of this is spent on direct sums in \eqref{hcmsnear} and \eqref{fmmcorr}
(the fine FMM in \eqref{fmmcorr} takes only 1.5 s), as well as geometry
and inevitable \matlab\ overheads.
We found this to be 50 times faster than
directly applying the $O(N^2)$ formulae from section~\ref{s:implhelm},
justifying the utility of our proposed $O(N)$ scheme.
Fig.~\ref{f:high}(d) plots a zoom of the resulting relative error.

\begin{rmk}[Reference solution]
In previous examples $u$ was analytically known.
In order to assess the error in $\hat u$ for this example
we compute a reference $u$ in the following expensive fashion.
Outside $\Omega_\tbox{bad}$ we used the FMM from the density $\{\tau_j\}_{j=1}^N$
refined by a factor 10 by trigonometric interpolation.
Inside $\Omega_\tbox{bad}$ we did the same but with a factor $10^3$,
apart from the few hundred points in the narrow ribbon inside
$\Gamma_{10^{-3}\al_\tbox{bad}}$. For these last points we used 9th-order polynomial
extrapolation from points distances $\{2^j(10^{-2}\pi/N)|Z'(s)|\}_{j=0}^9$
along the normal direction.
This appears to give around 12 digits; we are not able
reliably to get more.
\end{rmk}

\begin{rmk}[Convex vs concave]
The largest errors occur at the corners
of boxes near highly-convex parts of the boundary;
we believe that this is due to the
nearby interior Schwarz singularity of $\pO$, which
generically induces a singularity in the Helmholtz continuation of the
solution $u$ \cite{Mi86} (briefly reviewed in \cite[Sec.~3.1]{mfs}).
The latter in turn slows the $p$-convergence of the expansion.
By contrast, concave parts have Schwarz singularities on the physical side of
$\pO$ which thus do not affect $u$.
We believe this explains why lower accuracy is reported at
convex (but not concave) locations in QBX \cite{qbx}.
\end{rmk}

\section{Conclusion and discussion}
\label{s:conc}

Firstly, we analyzed the spatial distribution of
the error in evaluating Laplace layer potentials
using the popular global periodic trapezoid rule.
The key tools were generalizations of the Davis theorem,
and the annular conformal map between the complex parameter plane and
the physical plane.
We found (Theorems~\ref{t:convd} and \ref{t:convs})
that the exponential convergence rate at a point is
simply the {\em imaginary part of its preimage under this map},
a result believed to be new.
Error contours are thus given by ``imaginary parameter translations''
of the boundary; these sweep out an annular neighborhood $\Omega_\tbox{bad}$
where errors are unacceptable.
Empirically, errors are similar in the Helmholtz case.

Secondly, we devised a surrogate local expansion method
for accurate evaluation in $\Omega_\tbox{bad}$.
Our main analytical result is its exponential convergence
(and hence that of QBX \cite{qbx})
in the analytic Laplace case (Theorems~\ref{t:eps} and \ref{t:epss}).
The scheme can be implemented via the FMM in $O(N)$ time,
and generalizes easily to the Helmholtz equation,
as we showed in a challenging high-frequency scattering application.
Our scheme gives errors close to machine precision, given density values at
only the number $N$ of nodes sufficient for the \Ny\ method to converge.
Our experiments also highlighted the need for more
understanding of the \Ny\ exponential convergence rate
and of the situations in which
the trigonometric interpolant is inferior to the \Ny\ one.

\begin{rmk}[Adaptivity]
We used global quadrature in this study;
an adaptive panel-based (composite) underlying quadrature, however, would be
preferred in a production code, and would allow refinement at corners.
The surrogate scheme is easy to implement with Gaussian panels,
needing only one Lagrange interpolation to the fine nodes.
Furthermore, the fast scheme of section~\ref{s:ON} becomes simpler:
the only FMM needed is the native evaluation $u^{(N)}$,
corrected in $O(N)$ effort by refining only the 3 nearest panels.
Comparison against the recent panel-based scheme of Helsing \cite{helsingtut}
would be desired.
\label{r:panels}
\end{rmk}

We expect the generalization to surfaces in $\R^3$ to be fruitful,
both to evaluate close to a surface, and to build
singular \Ny\ quadratures on the surface
(generalizing the QBX scheme in $\R^2$ \cite{qbx}).
Since only smooth interpolation, the local expansion, and the
addition theorem are needed,
this is simpler to implement than most existing high-order schemes
\cite{coltonkress,brunoFMM,ying06,bremer3d}.
Indeed, since initial submission of the present work,
numerical results in $\R^3$ have been encouraging \cite{qp3dqbx},
and an analysis of QBX,
including the Helmholtz single-layer potential in $\R^3$,
has appeared \cite{QBX2}.
In $\R^3$, since one can no longer
exploit a simple link between Laplace solutions and holomorphic functions,
the analysis relies on estimates for spherical harmonics.
We suggest that an analysis for the present work in the case of the Helmholtz
equation in $\R^2$
should be tractable via Vekua's map from holomorphic functions
to Helmholtz solutions \cite{vekua,He57,timothesis,moiola11}.




\section*{Acknowledgments}
The author is very grateful to Hanh Nguyen, via the support of
the Women in Science Project (WISP) at Dartmouth College,
for testing a Laplace double-layer implementation in the spring of 2011.
The author also benefited from
discussions with Stephen Langdon, Zydrunas Gimbutas, Leslie Greengard,
Andreas Kl\"ockner, Mike O'Neil, and Nick Trefethen.
This work is supported through the National Science Foundation via
grants DMS-0811005 and DMS-1216656.

\appendix
\section{Code}
While we have not yet
released a formal package implementing the methods of this
paper, we make available codes that generated the figures
here: \\
\mbox{}\qquad{\tt http://math.dartmouth.edu/$\sim$ahb/software/closeeval}\\
These require \matlab\ \cite{matlab},
\mpspack\ version at least 1.32 \cite{mpspack},
and {\tt FMMLIB2D} version 1.2 \cite{hfmm2d}.

\bibliography{alex}

\begin{thebibliography}{10}

\bibitem{atkinson}
{\sc K.~Atkinson}, {\em The numerical solution of integral equations of the
  second kind}, Cambridge University Press, 1997.

\bibitem{jeon98}
{\sc K.~Atkinson and Y.-M. Jeon}, {\em Algorithm 788: {A}utomatic boundary
  integral equation programs for the planar {L}aplace equation}, ACM Trans.
  Math. Software, 24 (1998), pp.~395--417.

\bibitem{mfs}
{\sc A.~H. Barnett and T.~Betcke}, {\em Stability and convergence of the
  {M}ethod of {F}undamental {S}olutions for {H}elmholtz problems on analytic
  domains}, J. Comput. Phys., 227 (2008), pp.~7003--7026.

\bibitem{mpspack}
\leavevmode\vrule height 2pt depth -1.6pt width 23pt, {\em {\verb+MPSpack+}: A
  {MATLAB} toolbox to solve {H}elmholtz {PDE}, wave scattering, and eigenvalue
  problems}, 2008--2012.
\newblock \verb+http://code.google.com/p/mpspack/+.

\bibitem{qplp}
{\sc A.~H. Barnett and L.~Greengard}, {\em A new integral representation for
  quasi-periodic fields and its application to two-dimensional band structure
  calculations}, J. Comput. Phys., 229 (2010), pp.~6898--6914.

\bibitem{qp3dqbx}
{\sc A.~H. Barnett, L.~Greengard, and Z.~Gimbutas}, {\em Efficient and robust
  integral equation methods for acoustic scattering from doubly-periodic media
  in three dimensions}, 2013.
\newblock in preparation.

\bibitem{timothesis}
{\sc T.~Betcke}, {\em Computations of Eigenfunctions of Planar Regions}, PhD
  thesis, Oxford University, UK, 2005.

\bibitem{bremer3d}
{\sc J.~Bremer and Z.~Gimbutas}, {\em A {Nystr\"om} method for weakly singular
  integral operators on surfaces}, J. Comput. Phys., 231 (2012),
  pp.~4885--4903.

\bibitem{brunoFMM}
{\sc O.~P. Bruno and L.~A. Kunyansky}, {\em A fast, high-order algorithm for
  the solution of surface scattering problems: basic implementation, tests, and
  applications}, J. Comput. Phys., 169 (2001), pp.~80--110.

\bibitem{bathreadingrev}
{\sc S.~N. Chandler-Wilde, I.~G. Graham, S.~Langdon, and E.~A. Spence}, {\em
  Numerical-asymptotic boundary integral methods in high-frequency acoustic
  scattering}, Acta Numer.,  (2012), pp.~89--305.

\bibitem{coltonkress}
{\sc D.~Colton and R.~Kress}, {\em Inverse acoustic and electromagnetic
  scattering theory}, vol.~93 of Applied Mathematical Sciences,
  Springer-Verlag, Berlin, second~ed., 1998.

\bibitem{davis59}
{\sc P.~J. Davis}, {\em On the numerical integration of periodic analytic
  functions}, in Proceedings of a Symposium on Numerical Approximations, R.~E.
  Langer, ed., University of Wisconsin Press, 1959.

\bibitem{Da74}
{\sc P.~J. Davis}, {\em The {S}chwarz function and its applications}, The
  Mathematical Association of America, Buffalo, N. Y., 1974.
\newblock The Carus Mathematical Monographs, No. 17.

\bibitem{davisrabin}
{\sc P.~J. Davis and P.~Rabinowitz}, {\em Methods of Numerical Integration},
  Academic Press, San Diego, 1984.

\bibitem{QBX2}
{\sc C.~L. Epstein, L.~Greengard, and A.~Kl\"ockner}, {\em On the convergence
  of local expansions of layer potentials}, 2012.
\newblock {\tt arXiv:1212.3868}.

\bibitem{gedney03}
{\sc S.~D. Gedney}, {\em On deriving a locally corrected {Nystr\"om} scheme
  from a quadrature sampled moment method}, IEEE Trans. Antennas Propag., 51
  (2003), pp.~2402--2412.

\bibitem{hfmm2d}
{\sc Z.~Gimbutas and L.~Greengard}, {\em {FMMLIB2D}, {F}ortran libraries for
  fast multipole method in two dimensions}, 2011.
\newblock {\tt http://www.cims.nyu.edu/cmcl/fmm2dlib/fmm2dlib.html}.

\bibitem{hao}
{\sc S.~Hao, A.~H. Barnett, P.~G. Martinsson, and P.~Young}, {\em High-order
  accurate {N}ystr{\"o}m discretization of integral equations with weakly
  singular kernels on smooth curves in the plane}, 2013.
\newblock accepted, Adv.\ Comput.\ Math., {\tt arxiv:1112.6262v2}.

\bibitem{helsingtut}
{\sc J.~Helsing}, {\em Solving integral equations on piecewise smooth
  boundaries using the {RCIP} method: a tutorial}, 2012.
\newblock preprint, 34 pages, {\tt arXiv:1207.6737v3}.

\bibitem{helsing_close}
{\sc J.~Helsing and R.~Ojala}, {\em On the evaluation of layer potentials close
  to their sources}, J. Comput. Phys., 227 (2008), pp.~2899--2921.

\bibitem{He57}
{\sc P.~Henrici}, {\em A survey of {I}. {N}. {V}ekua's theory of elliptic
  partial differential equations with analytic coefficients}, Z. Angew. Math.
  Phys., 8 (1957), pp.~169--203.

\bibitem{hunter71}
{\sc D.~B. Hunter}, {\em The evaluation of integrals of periodic analytic
  functions}, BIT Numer. Math., 11 (1971), pp.~175--180.

\bibitem{qbx}
{\sc A.~Kl\"ockner, A.~H. Barnett, L.~Greengard, and M.~O'Neil}, {\em
  Quadrature by expansion: a new method for the evaluation of layer
  potentials}, 2013.

\bibitem{bookrevshapiro}
{\sc J.~Korevaar}, {\em Book review of ``{T}he {S}chwarz function and its
  generalization to higher dimensions'' by {H.} {S.} {S}hapiro}, Bull. Amer.
  Math. Soc., 31 (1994), pp.~112--116.

\bibitem{kress91}
{\sc R.~Kress}, {\em Boundary integral equations in time-harmonic acoustic
  scattering}, Mathl. Comput. Modelling, 15 (1991), pp.~229--243.

\bibitem{LIE}
{\sc R.~Kress}, {\em Linear Integral Equations}, vol.~82 of Applied
  Mathematical Sciences, Springer, second~ed., 1999.

\bibitem{mayo}
{\sc A.~Mayo}, {\em Fast high order accurate solution of {L}aplace's equation
  on irregular regions}, SIAM J. Sci. Stat. Comput., 6 (1985), pp.~144--157.

\bibitem{mckenney}
{\sc A.~McKenney}, {\em An adaptation of the fast multipole method for
  evaluating layer potentials in two dimensions}, Computers Math.\ Applic., 31
  (1996), pp.~33--57.

\bibitem{Mi86}
{\sc R.~F. Millar}, {\em Singularities and the {R}ayleigh hypothesis for
  solutions to the {H}elmholtz equation}, IMA J. Appl. Math., 37 (1986),
  pp.~155--171.

\bibitem{moiola11}
{\sc A.~Moiola, R.~Hiptmair, and I.~Perugia}, {\em {V}ekua theory for the
  {H}elmholtz operator}, Z.\ Angew.\ Math.\ Phys., 62 (2011), pp.~779--807.

\bibitem{dlmf}
{\sc F.~W.~J. Olver, D.~W. Lozier, R.~F. Boisvert, and C.~W. Clark}, eds., {\em
  {NIST} Handbook of Mathematical Functions}, Cambridge University Press, 2010.
\newblock {\tt http://dlmf.nist.gov}.

\bibitem{schwabsurf}
{\sc C.~Schwab and W.~L. Wendland}, {\em On the extraction technique in
  boundary integral equations}, Math. Comp., 68 (1999), pp.~91--122.

\bibitem{steinshakarchi}
{\sc E.~M. Stein and R.~Shakarchi}, {\em Complex Analysis (Princeton Lectures
  in Analysis, No. 2)}, Princeton University Press, 2003.

\bibitem{matlab}
{\sc {The MathWorks, Inc.}}, {\em {MATLAB} software}, Copyright (c) 1984--2012.
\newblock {\tt http://www.mathworks.com/matlab}.

\bibitem{ATAP}
{\sc L.~N. Trefethen}, {\em Approximation Theory and Approximation Practice},
  SIAM, 2012.
\newblock {\tt http://www.maths.ox.ac.uk/chebfun/ATAP}.

\bibitem{vekua}
{\sc I.~N. Vekua}, {\em Novye metody rezhenija elliptickikh uravnenij, (OGIZ,
  Moscow and Leningrad, 1948); English translation: New methods for solving
  elliptic equations}, North-Holland, 1967.

\bibitem{ying06}
{\sc L.~Ying, G.~Biros, and D.~Zorin}, {\em A high-order {3D} boundary integral
  equation solver for elliptic {PDE}s in smooth domains}, J. Comput. Phys., 216
  (2006), pp.~247--275.

\end{thebibliography}
\end{document}